\documentclass{cmslatex}
\usepackage[paperwidth=7in, paperheight=10in, margin=.875in]{geometry}
\usepackage[backref,colorlinks,linkcolor=red,anchorcolor=green,citecolor=blue]{hyperref}
\usepackage{amsfonts,amssymb}
\usepackage{cite}
\renewcommand{\theequation}{\arabic{section}.\arabic{equation}}

\usepackage{graphicx}
\usepackage{epstopdf}
\usepackage{tikz}
\usetikzlibrary{shapes,arrows,matrix,patterns,positioning}
\usepackage{color}
\usepackage{amsmath,bm,color,epsfig,enumerate,caption}
\usepackage{hyperref}
\usepackage{booktabs}
\usepackage{multirow}
\usepackage{array}
\usepackage[ruled,linesnumbered,algo2e]{algorithm2e} 
\usepackage{siunitx}
\usepackage{subcaption}
\usepackage{comment}

\newcommand{\norm}[1]{\left\lVert#1\right\rVert}
\newcommand{\abs}[1]{\left\lvert#1\right\rvert}

\newcommand{\diff}{\,\text{d}}

\newcommand{\calI}{\mathcal{I}}

\newcommand{\calP}{\mathcal{P}}
\newcommand{\calR}{\mathcal{R}}

\newcommand{\bbF}{\mathbb{F}}
\newcommand{\bbC}{\mathbb{C}}
\newcommand{\bbR}{\mathbb{R}}

\newcommand{\eps}{\epsilon}

\newcommand{\sn}{\text{ sn}}
\newcommand{\cn}{\text{ cn}}
\newcommand{\sign}[1]{\text{sign}\left(#1\right)}

\newcommand{\nlam}{{n_\lambda}}

\newcommand{\hZ}{\widehat{Z}}
\newcommand{\hM}{\widehat{M}}
\newcommand{\ha}{\widehat{a}}
\newcommand{\hc}{\widehat{c}}
\newcommand{\hsigma}{\widehat{\sigma}}
\newcommand{\hw}{\widehat{w}}

\makeatletter
\newcommand*{\extendadd}{
    \mathbin{
        \mathpalette\extend@add{}
    }
}
\newcommand*{\extend@add}[2]{
    \ooalign{
        $\m@th#1\leftrightarrow$%
        \vphantom{$\m@th#1\updownarrow$}
        \cr
        \hfil$\m@th#1\updownarrow$\hfil
    }
}
\makeatother

   \allowdisplaybreaks
\begin{document}
\title{Interior Eigensolver for Sparse Hermitian Definite Matrices Based on Zolotarev's Functions}

\author{Yingzhou Li$^\sharp$,
        Haizhao Yang$^\dagger$,
        \vspace{0.1in}\\
        $\sharp$ School of Mathematical Sciences, Fudan University\\
        $\dagger$ Department of Mathematics, Purdue University
}

          \author{Yingzhou Li\thanks{School of
          Mathematical Sciences, Fudan University,
          (yingzhouli0417@gmail.com). \url{https://yingzhouli.com/}}
          \and Haizhao Yang \thanks{Department
          of Mathematics, Purdue University,
          (haizhao@purdue.edu). \url{https://haizhaoyang.github.io/}}}

         \pagestyle{myheadings} \markboth{Interior Eigensolver Based on Zolotarev's Functions}{Yingzhou Li and Haizhao Yang} \maketitle

          \begin{abstract}
This paper proposes an efficient method for computing selected generalized
eigenpairs of a sparse Hermitian definite matrix pencil $(A,B)$. Based on
Zolotarev's best rational function approximations of the signum function and
conformal mapping techniques, we construct the best rational function
approximation of a rectangular function supported on an arbitrary interval
via function compositions with partial fraction representations. This new
best rational function approximation can be applied to construct spectrum
filters of $(A,B)$ with a smaller number of poles than a direct construction
without function compositions. Combining fast direct solvers and the
shift-invariant generalized minimal residual method, a hybrid fast algorithm
is proposed to apply spectral filters efficiently.  Compared to the
state-of-the-art algorithm FEAST, the proposed rational function
approximation is more efficient when sparse matrix factorizations are
required to solve multi-shift linear systems in the eigensolver, since the
smaller number of matrix factorizations is needed in our method.  The
efficiency and stability of the proposed method are demonstrated by
numerical examples from computational chemistry.
          \end{abstract}


\section{Introduction}
\label{sec:intro}

Given a sparse Hermitian definite matrix pencil $(A,B)$ (i.e., $A$ and $B$
are Hermitian and $B$ is positive-definite) in $\bbF^{N\times N}$, where
$\bbF=\bbR$ or $\bbF=\bbC$, and an interval $(a, b)$ of interest, this paper
aims at identifying all the eigenpairs $\{(\lambda_j,x_j)\}_{1\leq j\leq
\nlam}$ \footnote{ Through out the paper, we assume that the exact number of
eigenvalues, $\nlam$, is known a priori, which would simplify the
presentation of the method. While, in practice, an estimated number of
$\nlam$ is enough for the algorithm.} of $(A, B)$ in $(a, b)$, i.e.,
\begin{equation}
    Ax_j = \lambda_j B x_j \quad \text{and} \quad a < \lambda_j < b,
    \quad j = 1,2,\dots,\nlam.
    \label{eq:IGEVP}
\end{equation}
The interior generalized eigenvalue problem not only can be applied
to solve the full generalized eigenvalue problem via the spectrum
slicing idea~\cite{Aktulga2014, VanBarel2016, Li2016Xi, Polizzi2009,
Sakurai2003, Sakurai2007, Schofield2012, Xi2016, Ye2016}, but also
is a stand-alone problem encountered in many fields in science
and engineering (such as computational chemistry, control theory,
material science, etc.), where a partial spectrum is of interest.

\subsection{Related Work}

A powerful tool for solving the interior generalized eigenvalue
problem is the subspace iteration method accelerated by spectrum
filters. Let $P_{ab}(A,B)$ be an approximate spectrum projector onto
the eigen-subspace of the matrix pencil $(A, B)$ corresponding to
the desired eigenvalues in $(a, b)$.
%
%
%
%
A possible way to construct $P_{ab}(A,B)$  is to design a filter function $R_{ab}(x)$ as a good
approximation to a rectangular function with
a support on $(a,b)$ (denoted as $S_{ab}(x)$), and define $P_{ab}(A,B)
= R_{ab}(B^{-1}A)$. There are mainly two kinds of filter functions:
polynomial filters \cite{Li2016Xi, Schofield2012} and rational filters
\cite{Aktulga2014, VanBarel2016, VanBarel2016Kravanja, Polizzi2009,
Sakurai2003, Sakurai2007, Xi2016, Ye2016}. The
difficulty in designing an appropriate filter comes from the dilemma
that: an accurate approximation to the spectrum projector requires
a polynomial of high degree or a rational function with many poles;
however this in turn results in expensive computational cost in applying
the spectrum projector $R_{ab}(B^{-1}A)$.

In general, a rational filter can be written as follows
\begin{equation}\label{eq:rationalfilter}
    R_{ab}(x)=\alpha_0+\sum_{j=1}^p \frac{\alpha_j}{x-\sigma_j},
\end{equation}
where $\{\alpha_j\}_{0 \leq j \leq p}$ are weights, $\{\sigma_j\}_{1
\leq j \leq p}$ are poles, and $p$ is the number of poles. Hence,
applying the spectrum projector $R_{ab}(B^{-1}A)$ to a vector $v$
requires solving $p$ linear systems $\{(A-\sigma_j B)^{-1}Bv\}_{1\leq
j\leq p}$.  Therefore, a large number $p$ makes it expensive to apply
the approximate spectrum projector $R_{ab}(B^{-1}A)$.  A natural idea
is to solve the linear systems $\{(A-\sigma_j B)^{-1}Bv\}_{1\leq j\leq
p}$ in parallel.  However, for the purposes of energy efficiency and
numerical stability, an optimal $p$ is always preferred. Extensive
effort has been made to develop rational functions with $p$ as small
as possible while keeping the accuracy of the approximation.

\begin{figure}[htb]
    \centering
    \begin{subfigure}[t]{0.48\textwidth}
        \centering
        \includegraphics[width=\textwidth]{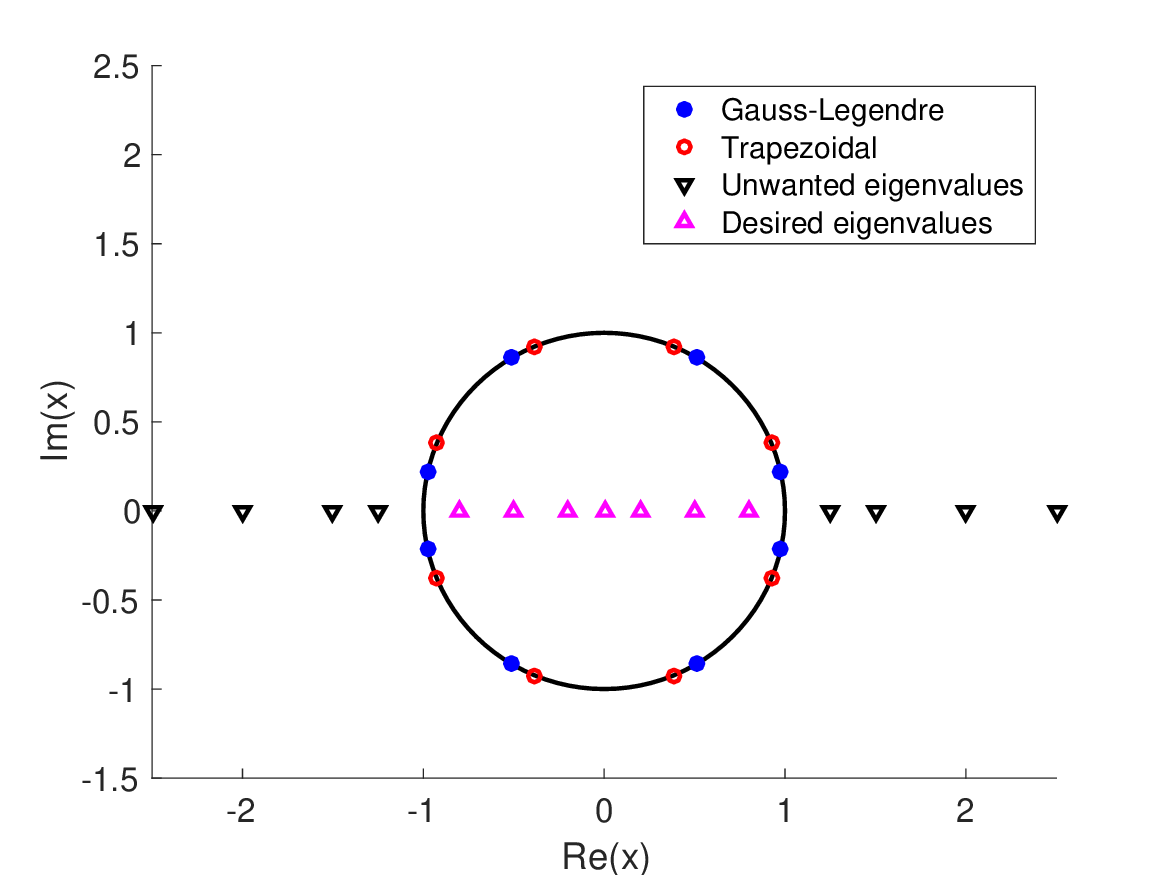}
        \caption{}
    \end{subfigure}
    ~~
    \begin{subfigure}[t]{0.48\textwidth}
        \centering
        \includegraphics[width=\textwidth]{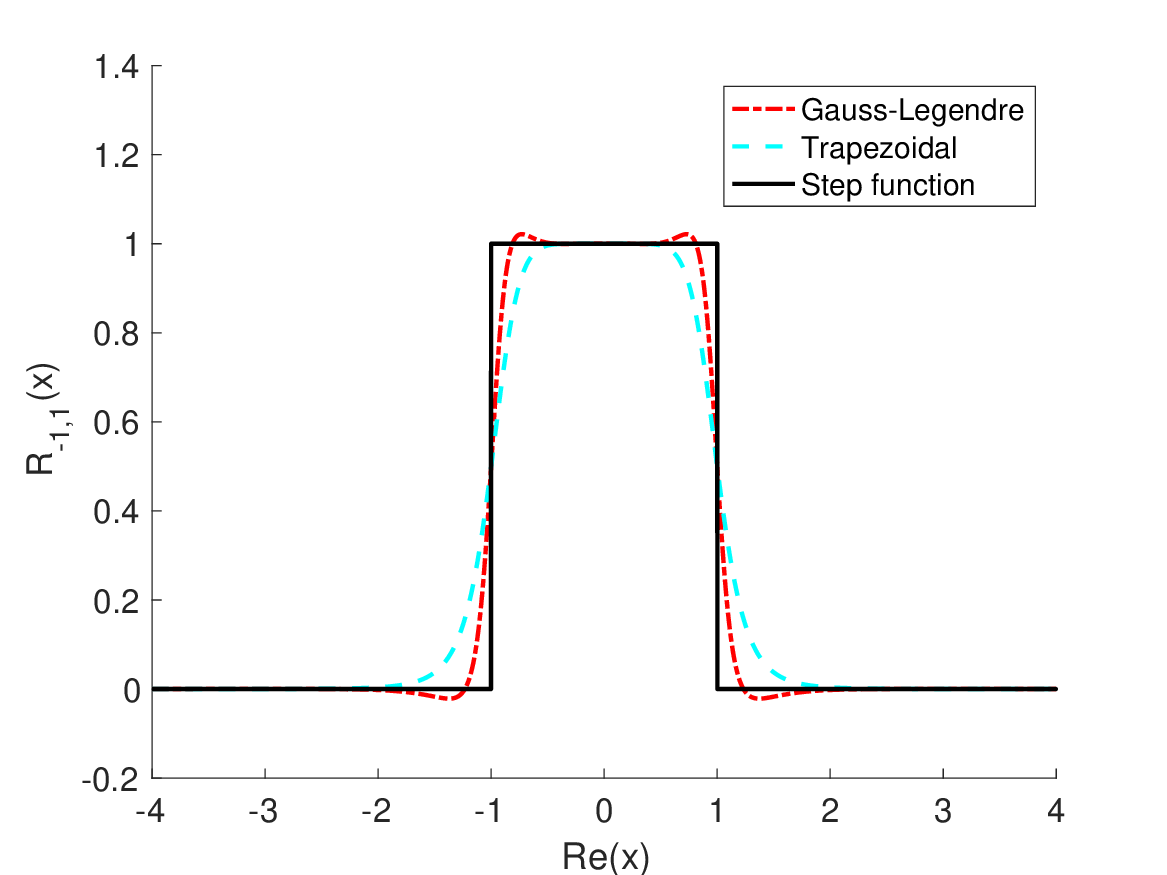}
        \caption{}
    \end{subfigure}

    \caption{(a) An example of a unit circle contour $\Gamma$ centered at
    the origin, i.e., the desired spectrum range is $(-1,1)$, together
    with eight Gauss-Legendre quadrature points (solid blue circle)
    and eight Trapezoid quadrature points (red circle).  The desired
    eigenvalues (pink up triangular) are inside the contour whereas
    the unwanted eigenvalues (black down triangular) are outside. (b)
    A rectangular function in solid black line  with rational
    functions corresponding to the quadratures from (a).}
    \label{fig:1}

\end{figure} 

Many rational filters in the literature were constructed by discretizing
the contour integral on the complex plane,
\begin{equation}
    \pi(x) = \frac{1}{2\pi \imath} \oint_{\Gamma} \frac{1}{x-z} \diff z,
    \quad x \notin \Gamma \label{eq:contourintegral}
\end{equation}
with an appropriate quadrature rule (e.g., the Gauss-Legendre quadrature
rule~\cite{Polizzi2009}, the trapezoidal quadrature rule~\cite{Tang2014,
Ye2016}, and the Zolotarev quadrature rule~\cite{Guttel2015}). Here
$\Gamma$ is a closed contour on the complex plane intersecting the
real axis at $z=a$ and $z=b$ with all desired eigenvalues inside
$(a,b)$ and other eigenvalues outside (See Figure \ref{fig:1}
(left) for an example). Suppose $\{\sigma_j\}_{1 \leq j \leq p}$ and
$\{w_j\}_{1 \leq j \leq p}$ are the quadrature points and weights in
the discretization of the contour $\Gamma$, respectively, the contour
integral \eqref{eq:contourintegral} is discretized as a rational function
\begin{equation}
    R(x) = \sum_{j=1}^{p} \frac{w_j}{2 \pi \imath (x-\sigma_j)} =
   \alpha_0+ \sum_{j=1}^p  \frac{\alpha_j}{x-\sigma_j},
\end{equation}
where $\alpha_0 = 0$, and $\alpha_j = \frac{w_j}{2 \pi \imath}$
for $j = 1, 2, \dots, p$. Some other methods advanced with conformal
maps~\cite{Hale2008, Lin2009} and optimization~\cite{VanBarel2016,
Xi2016} can also provide good rational filters.

\subsection{Contribution}

Based on Zolotarev's best rational function approximations of the signum
function and conformal maps, we construct the best rational function
$R_{ab}(x)$ approximating a rectangular function supported on an arbitrary
interval $(a,b)$. The optimality in this paper is in terms of the uniform
approximation error among the class of rational functions of the same type.
Combining fast direct solvers and the shift-invariant generalized minimal
residual method~(GMRES), a hybrid fast algorithm is proposed to apply the
spectrum filter $R_{ab}(B^{-1}A)$ to given vectors.

Suppose $a \in (a_-, a_+)$ and $b \in (b_-, b_+)$ respectively, and no
eigenvalue lies in $(a_-, a_+)$ and $(b_-, b_+)$.  The proposed rational
filter $R_{ab}(x)$ is constructed via the composition of Zolotarev's
functions as follows
\begin{equation} \label{eq:Rab}
    R_{ab}(x) = \frac{ Z_{2r}( \hZ_{2r}( T(x); \ell_1 ); \ell_2 ) +1
    }{2},
\end{equation}
where $Z_{2r}( x; \ell )$ is the Zolotarev's function of type $(2r-1,2r)$,
$\hZ_{2r}( x; \ell )$ is the scaled Zolotarev's function
\begin{equation}
    \hZ_{2r}( x; \ell ) = \frac{ Z_{2r}( x; \ell ) }{ \max_{ x \in [
    \ell, 1 ] }{Z}_{2r}( x; \ell )},
\end{equation}
and $T(x)$ is a M\"obius transformation of the form
\begin{equation}
    T(x) = \gamma \frac{ x - \alpha }{ x - \beta }
\end{equation}
with $\alpha \in ( a_-, a_+ )$ and $\beta \in ( b_-, b_+ )$ such that
\begin{equation}
    T( a_- ) = -1, \quad T( a_+ ) = 1, \quad T( b_- ) = \ell_1, \quad
    \text{ and } T( b_+ ) = -\ell_1.
\end{equation}
In the above construction, the variables $\alpha$, $\beta$, $\gamma$,
$\ell_1$, and $\ell_2$ are all determined by $a_-$, $a_+$, $b_-$,
and $b_+$.

The novelty of the proposed rational filter in \eqref{eq:Rab} is to
construct a high-order rational function for an arbitrary interval via the
composition of two Zolotarev's functions and a M\"obius transformation. This
new construction can significantly improve the approximation accuracy for a
rectangular function approximation even if $r$ is small, as compared to
other methods via a single Zolotarev's functions in \cite{Guttel2015}.
Similar composition ideas have been applied to the signum function
approximation (e.g., polar decomposition of matrices \cite{Nakatsukasa2010},
full diagonalization of matrices \cite{Nakatsukasa2016}, and the density
matrix purification \cite{Mazziotti2003, Niklasson2002, Palser1998}), and
the square root function approximation for accelerating Heron?s iteration
\cite{Braess1984, Ninomiya1970}. After the completion of the investigation
described in this paper the authors became aware of a work \cite{Guttel2015}
that addresses a similar question about the optimal rational filter via
Zolotarev's functions. The main difference is that: we propose to construct
high-order rational functions via function compositions, while
\cite{Guttel2015} directly constructs the rational function without
compositions. Function composition can reduce the number of direct matrix
factorizations needed in the computation and hence would reduce the
computational time. This idea has not been explored yet for computing
interior eigenpairs and is the main contribution of our paper.

An immediate challenge arises from applying the composition of functions
$R_{ab}(B^{-1}A)$ in \eqref{eq:Rab} to given vectors when $A$ and
$B$ are sparse matrices. Directly computing $R_{ab}(B^{-1}A)$ will
destroy the sparsity of $A$ and $B$ since $R_{ab}(B^{-1}A)$ is dense.
Fortunately, the function composition structure in \eqref{eq:Rab}
admits a hybrid fast algorithm for the matrix-vector multiplication
(matvec) $R_{ab}(B^{-1}A)V$, where $A$ and $B$ are sparse Hermitian
matrices of size $N$ with $O(N)$ nonzero entries, $B$ is positive
definite, and $V$ is a tall skinny matrix of size $N$ by $O(1)$. We
apply the multifrontal method~\cite{Duff1983,Liu1992} to solve the
sparse linear systems involved in $\hZ_{2r}( T( B^{-1} A ); \ell_1 )V$.
The multifrontal method consists of two parts: the factorization
of sparse matrices and the application of the factorization. Once
sparse factors have been constructed, evaluating $\hZ_{2r}( T( B^{-1}
A ); \ell_1 )V$ is efficient; in this sense, the multifrontal method
converts the dense matrix $\hZ_{2r}( T( B^{-1} A ); \ell_1 )$ into
an operator with a fast application. Since the Zolotarev's function
well approximates the signum function, the matrix $\hZ_{2r}( T( B^{-1}
A ); \ell_1 )$ has a condition number close to $1$.  Therefore, the
computation $Z_{2r}( \hZ_{2r}(T( B^{-1} A ); \ell_1 ); \ell_2 ) V$
can be carried out efficiently by the GMRES iterative method. As we
shall see later, by the shift-invariant property of Krylov subspace,
the computational time can be further reduced in the GMRES.

When we incorporate the above hybrid fast algorithm into the subspace
iteration method, the factorization time of the multifrontal method can
be treated as precomputation, since all the multi-shift linear systems
in every iteration remain unchanged. Since $Z_{2r}( \hZ_{2r}(T( B^{-1}
A ); \ell_1 ); \ell_2 )$ is able to approximate the desired spectrum
projector of $(A,B)$ accurately, the subspace iteration method usually
needs only one or two iterations to identify desired eigenpairs up
to an $10^{-10}$ relative error. Hence, the dominant computing time
in the proposed interior eigensolver is the factorization time in the
multifrontal method.

\subsection{Organization}

In what follows, we introduce the subspace iteration, the best rational
filter, and the hybrid fast algorithm in Section \ref{sec:ss}. In Section
\ref{sec:results}, extensive numerical examples of a wide range of sparse
matrices are presented to demonstrate the efficiency of the proposed
algorithms. Finally, we conclude this paper with a short discussion in
Section \ref{sec:discussion}.

\section{Algorithm}
\label{sec:ss}

First, we recall a standard subspace iteration accelerated by a
rational filter for interior generalized eigenvalue problems in
Section~\ref{sub:view}. Second, we introduce the best rational filter
$R_{ab}(x)$ in \eqref{eq:Rab} and show its efficiency of approximating
the rectangular function $S_{ab}(x)$ on the interval
\begin{equation}
    ( -\infty, a_- ] \cup [ a_+, b_- ] \cup [ b_+, \infty ),
\end{equation}
where $(a_-,a_+)$ and $(b_-,b_+)$ are eigengaps around $a$ and $b$,
i.e., there is no eigenvalue inside these two intervals. Third, the
hybrid fast algorithm for evaluating the matvec $R_{ab}(B^{-1}A)V$
is introduced in Section \ref{sub:matvec}.

\begin{table}[htp]
    \centering
    \begin{tabular}{l|l}
        \toprule
        Notation & Description \\
        \toprule
        $N$ & Size of the matrix \\
        $\bbF$ & Either $\bbR$ or $\bbC$ \\
        $A, B$ & Sparse Hermitian definite matrix of size $N \times N$ \\
        $(A,B)$ & Matrix pencil \\
        $(a,b)$ & Interval of interest on the spectrum of $(A,B)$ \\
        $(a_-,a_+), (b_-,b_+)$ & Eigengaps around $a$ and $b$ respectively
        \\
        $\nlam$ & Number of eigenvalues in the interval \\
        $k$ & Oversampling constant \\
        \bottomrule
    \end{tabular}
    \caption{Commonly used notations.}
    \label{tab:notations}
\end{table}

Throughout this paper, we adopt MATLAB notations for submatrices and
indices.  Besides usual MATLAB notations, we summarize a few notations
that would be used in the rest of the paper without further explanation
in Table~\ref{tab:notations}.

\subsection{Subspace iteration with rational filters}\label{sub:view}

Various subspace iteration methods have been proposed and analyzed
in the literature. For the completeness of the presentation,
we introduce a standard one in conjunction of a rational filter in 
Algorithm~\ref{alg:subiter} 
for the interior generalized eigenvalue problem for a matrix pencil
$(A,B)$ on a spectrum interval $(a,b)$.

\begin{algorithm2e}[H]

\DontPrintSemicolon
\SetAlgoNoLine
\SetKwInOut{Input}{input}
\SetKwInOut{Output}{output}

    \Input{Sparse Hermitian matrix pencil $(A, B)$, a spectrum range
    $(a,b)$, the number of eigenpairs $\nlam$, and a rational filter 
    $R_{ab}(x)$}

    \Output{A diagonal matrix $\Lambda$ with diagonal entries being
    the eigenvalues of $(A,B)$ on $(a,b)$, $V$ are the corresponding
    eigenvectors}

    Generate orthonormal random vectors $Q \in \bbF^{N \times
    (\nlam+k)}$.

    \While{ not convergent\footnote{Locking can be applied in the iteration.} }{

        $Y = R_{ab}(B^{-1}A) Q$
 
	    Compute $\tilde{A}=Y^{*}AY$ and $\tilde{B}=Y^{*}BY$
 
	    Solve $\widetilde{A} \widetilde{Q} = \widetilde{\Lambda}
	    \widetilde{B} \widetilde{Q}$ for $\widetilde{\Lambda}$ and
	    $\widetilde{Q}$
 
        Update $Q =  Y \widetilde{Q}$

    }

    $\calI = \{ i \mid a < \widetilde{\Lambda}(i,i) < b \}$

    $\Lambda = \widetilde{\Lambda}(\calI, \calI)$

    $V = Q(:, \calI)$

\caption{A standard subspace iteration method}
\label{alg:subiter}

\end{algorithm2e}

The main cost in Algorithm~\ref{alg:subiter} is to compute $Y =
R_{ab}(B^{-1}A)Q$, since any other steps scale at most linearly in $N$
or even independent of $N$. If the rational function $R_{ab}(x)$ is not
a good approximation to the rectangular function $S_{ab}(x)$, it might
take many iterations for Algorithm~\ref{alg:subiter} to converge. Our
goal is to get an accurate rational function approximation $R_{ab}(x)$
so that only a small number of iterations is sufficient to estimate
the eigenpairs of $(A,B)$ with machine accuracy.  The method to achieve
the goal will be discussed in the next two subsections.

\subsection{Best rational filter by Zolotarev's functions}
\label{sub:high}


In what follows, we introduce basic definitions and theorems for
rational function approximations. Let $\calP_r$ denote the set of all
polynomials of degree $r$. A rational function $R(x)$ is said to be of
type $(r_1,r_2)$ if $R(x)=\frac{P(x)}{Q(x)}$ with $P(x) \in \calP_{r_1}$
and $Q(x) \in \calP_{r_2}$. We denote the set of all rational functions
of type $(r_1,r_2)$ as $\calR_{r_1,r_2}$. For a given function $f(x)$
and a rational function $R(x)$, the approximation error in a given
domain $\Omega$ is quantified by the infinity norm
\begin{equation}
    \norm{ f - R }_{L^\infty(\Omega)}=\sup_{x\in\Omega} \abs{f(x)-R(x)}.
\end{equation}
A common problem in the rational function approximation is the minimax
problem that identifies $R(x) \in \calR_{r_1,r_2}$ satisfying
\begin{equation}
    R = \arg \min_{g \in \calR_{r_1,r_2}} \norm{f-g}_{L^\infty(\Omega)}.
\end{equation}

More specifically, the minimax problem of interest for matrix computation is either
\begin{equation}
    R = \arg \min_{ g \in \calR_{2r-1,2r}} \norm{ \sign{x} - g(x)
    }_{L^\infty([-1,-\ell] \cup [\ell,1])},
    \label{eq:ZoloMinimax}
\end{equation}
where $r$ is a given integer and $\ell \in (0,1)$ is a given parameter, or
\begin{equation}
    R = \arg \min_{ g \in \calR_{(2r)^2,(2r)^2}} \norm{ S_{ab}(x) - g(x)
    }_{L^\infty((-\infty,a_-) \cup (a_+,b_-) \cup (b_+,\infty))},
    \label{eq:ZoloMinimax2}
\end{equation}
where $r$ is a given integer, $a_-<a$ and $a_+>a$ are two parameters around
$a$, $b_-<b$ and $b_+>b$ are two parameters around $b$. The problem in
\eqref{eq:ZoloMinimax} with $g$ of the particular type $\calR_{2r-1,2r}$,
has a unique solution and the explicit expression of the solution is given
by Zolotarev~\cite{Zolotarev1877}. We denote this best rational
approximation to the signum function by $Z_{2r}(x;\ell)$. To be more
precise, the following theorem summarizes one of Zolotarev's conclusions
which is rephrased by Akhiezer in Chapter $9$ in \cite{Akhiezer1990}, and by
Petrushev and Popov in Chapter $4.3$ in \cite{Petrushev1987}.
\begin{theorem}[Zolotarev's function]\label{thm:zoloSign}
    The best uniform rational approximant of type $(2r, 2r)$ for the
    signum function $\sign{x}$ on the set $[-1,-\ell] \cup [\ell,1]$,
    $0 < \ell < 1$, is given by
    \begin{equation}
        Z_{2r}(x; \ell) := M x \frac{ \prod_{j=1}^{r-1}(x^2+c_{2j})}
        {\prod_{j=1}^{r}(x^2+c_{2j-1})} \in \calR_{2r-1,2r},
        \label{eq:ZoloFunc}
    \end{equation}
    where $M>0$ is a unique constant such that
    \begin{equation}
        \min_{x \in [-1,-\ell]} Z_{2r}(x; \ell) + 1 = \min_{x \in
        [\ell,1]} Z_{2r}(x; \ell) - 1.
    \end{equation}
    The coefficients $c_1, c_2, \dots, c_{2r-1}$ are given by
    \begin{equation}\label{eqn:coef}
        c_j = \ell^2 \frac{ \sn^2 \left( \frac{jK'}{2r}; \ell' \right)}
        { \cn^2 \left( \frac{jK'}{2r}; \ell' \right)}, \quad j = 1, 2,
        \dots, 2r-1,
    \end{equation}
    where $\sn( x; \ell' )$ and $\cn( x; \ell' )$ are the Jacobi elliptic
    functions (see \cite{Akhiezer1956,Akhiezer1990}), $\ell' =
    \sqrt{1-\ell^2}$, and $K' = \int_0^{\pi/2} \frac{\diff \theta}{\sqrt{
    1 - (\ell')^2 \sin^2 \theta }}$.
\end{theorem}

By Add. E in \cite{Akhiezer1956}, the maximum approximation error
$\delta(2r,\ell):=\norm{ \sign{x} - Z_{2r}( x; \ell )}_{L^\infty}$
is attained at $2r+1$ points $x_1:=\ell<x_2<\dots<x_{2r}<x_{2r+1}:=1$
on the interval $[\ell,1]$ and also $2r+1$ points $x_{-j}:=-x_j$,
$j=1,2,\dots,2r+1$, on the interval $[-1,-\ell]$. The function
$\sign{x}-Z_{2r}(x;\ell)$ equioscillates between the $x_j$'s; in
particular,
\begin{equation}
    1 - Z_{2r}(x_j; \ell) = (-1)^{j+1} \delta(2r, \ell), \quad j = 1, 2,
    \dots, 2r+1.
\end{equation}

The approximation error of Zolotarev's functions as an approximant to
$\sign{x}$ decreases exponentially with degree $2r$ (\cite{Petrushev1987}
Section 4.3), i.e.
\begin{equation}
    \delta(2r,\ell)\approx C \rho^{-2r}
\end{equation}
for some positive $C$ and $\rho > 1$ that depends on $\ell$. In more
particular, Gon\v{c}ar~\cite{Goncar1969} gave the following quantitative
estimation on the approximation error, $\delta(2r,\ell)$:
\begin{equation}
    \frac{2}{ \rho^{2r} + 1 } \leq \delta(2r, \ell) \leq \frac{2}{
    \rho^{2r} - 1 },
\end{equation}
where
\begin{equation}\label{eqn:rho}
    \rho = \exp \left( \frac{ \pi K(\mu') }{ 4K(\mu)} \right),
\end{equation}
$\mu = \frac{ 1 - \sqrt{ \ell }}{ 1 + \sqrt{\ell} }$, $\mu' = \sqrt{ 1
- \mu^2 }$, and $K(\mu) = \int_0^{\pi/2} \frac{ \diff \theta }{ \sqrt{
1 - (\mu)^2 \sin^2 \theta }}$ is the complete elliptic integral of the
first kind for the modulus $\mu$.

Even though the approximation error $\delta(2r,\ell)$ decreases
exponentially in $r$, the decay rate of $\delta(2r,\ell)$ in $r$ might
still be slow if $\rho$ is small. In fact, $\rho$ could be small in
many applications when eigengaps are small. As we shall see later,
if the eigenvalues cluster together, $\ell$ should be very small and
hence $\rho$ is small by \eqref{eqn:rho}. As we have discussed earlier
in the introduction of this paper, it is not practical to use a large
$r$ due to the computational expense and numerical instability. This
motivates the study of the composition of Zolotarev's functions in
$\calR_{2r-1,2r}$, which constructs a high order Zolotarev's function
in $\calR_{(2r)^2-1,(2r)^2}$. Such a composition has a much smaller
approximation error
\begin{equation}
    \delta(4r^2,\ell)\approx C \rho^{-4r^2}.
\end{equation}

For simplicity, let us use the rescaled Zolotarev's function defined by
\begin{equation}\label{eqn:rZ}
    \hZ_{2r}(x; \ell) = \frac{ Z_{2r}(x; \ell) }{ \max_{x \in [\ell,1]}
    {Z}_{2r}(x; \ell) }.
\end{equation}
Note that $\max_{x \in [\ell,1]}\hZ_{2r}(x; \ell) = 1$, and
$\hZ_{2r}(x; \ell)$ maps the set $[-1, -\ell] \cup [\ell, 1]$ onto
$[-1, -\hZ_{2r}(\ell; \ell)] \cup [\hZ_{2r}(\ell; \ell), 1]$. Hence,
if one defines a composition via
\begin{equation}\label{eqn:S}
    S(x; \ell_1) = Z_{2r}( \hZ_{2r}(x; \ell_1); \ell_2),
\end{equation}
where $\ell_2=\hZ_{2r}(\ell_1; \ell_1)$, then $S(x; \ell_1) \in
\calR_{(2r)^2-1,(2r)^2}$ is the best uniform rational approximant
of type $((2r)^2, (2r)^2)$ for the signum function $\sign{x}$ on the
set $[-1,-\ell_1] \cup [\ell_1,1]$. This optimal approximation is an
immediate result of a more general theorem as follows.

\begin{theorem}\label{thm:comp}
    Let $\hZ_{2r_1}(x; \ell_1) \in \calR_{2r_1-1, 2r_1}$ be the rescaled
    Zolotarev's function corresponding to $\ell_1 \in (0,1)$, and
    $Z_{2r_2}(x; \ell_2) \in \calR_{2r_2-1, 2r_2}$ be the Zolotarev's
    function corresponding to $\ell_2 := \hZ_{2r_1}(\ell_1; \ell_1)$.
    Then
    \begin{equation}\label{eqn:Z}
        Z_{2r_2}( \hZ_{2r_1}( x; \ell_1); \ell_2 ) = Z_{(2r_1)(2r_2)}(x;
        \ell_1).
    \end{equation}
\end{theorem}

The proof of Theorem \ref{thm:comp} is similar to Theorem 3 in
\cite{Nakatsukasa2016}. Hence, we leave it to readers.

Finally, given a desired interval $(a,b)$ and the corresponding
eigengaps, $(a_-, a_+)$ and $(b_-, b_+)$, to answer the best rational function 
approximation in \eqref{eq:ZoloMinimax2}, we construct a uniform
rational approximant $R_{ab}(x)\in \calR_{(2r)^2, (2r)^2}$ via the
M\"obius transformation $T(x)$ as follows
\begin{equation}\label{eqn:R2}
    R_{ab}(x) = \frac{ S( T(x); \ell_1 ) + 1 }{2} = \frac{ Z_{2r}(
    \hZ_{2r}( T(x); \ell_1 ); \ell_2) + 1 }{2},
\end{equation}
where $\ell_2 = \hZ_{2r}( \ell_1; \ell_1 )$ and
\begin{equation}\label{mo}
    T(x) = \gamma \frac{ x - \alpha }{ x - \beta }
\end{equation}
with $\alpha \in (a_-,a_+)$ and $\beta \in (b_-, b_+)$ such that
\begin{equation}\label{eqn:req}
    T(a_-) = -1, \quad T(a_+) = 1, \quad T(b_-) = \ell_1, \quad \text{
    and } T(b_+) = -\ell_1.
\end{equation}
We would like to emphasize that the variables $\alpha$, $\beta$,
$\gamma$, $\ell_1$, and $\ell_2$ are determined by $a_-$, $a_+$,
$b_-$, and $b_+$ via solving the equations in \eqref{eqn:req} in the
above construction.  In practice, $a_-$, $a_+$, $b_-$, and $b_+$ can
be easily calculated from $a$ and $b$. We fixed the buffer region
$(a_-, a_+) \cup (b_-, b_+)$ first according to the eigengaps of
target matrices and construct a M\"obius transformation adaptive to
this region. This adaptive idea is natural but does not seem to have
been considered before in the literature. Following Corollary $4.2$ in
\cite{Guttel2015} on can easily prove that when $a_{-}=-\frac{1}{c}$,
$a_{+}=-c$, $b_-=c$, and $b_+=\frac{1}{c}$ for some $c>0$, the
rational function in \eqref{eqn:R2} is the best rational function
approximation to the step function $S_{-cc}(x)$ among all the rational
functions in $\{R(T(x)):R(x)\in\mathcal{R}_{(2r)^2-1,(2r)^2}\}\subset
\mathcal{R}_{(2r)^2,(2r)^2}$, where $T(x)$ is the M\"obius transformation
satisfying \eqref{eqn:req}. The following theorem shows that $R_{ab}(x)$
in \eqref{eqn:R2} is indeed the best uniform rational approximant of
type $((2r)^2, (2r)^2)$ for more general $a_-$, $a_+$, $b_-$, and $b_+$
among a larger class of rational functions.  \cite{Guttel2015} proved a similar theorem very briefly and our proof of Theorem \ref{thm:step} is different to that of \cite{Guttel2015}. The main purpose of our proof below is to make the paper self-contained.

%
%

\begin{theorem}\label{thm:step}
    The rational function $R_{ab}(x)$ given in \eqref{eqn:R2} satisfies
    the following properties:
    \begin{enumerate}[1)]
	    
	\item $R_{ab}(x)$ is the best uniform rational approximant of type
	$((2r)^2, (2r)^2)$ of the rectangular function $S_{ab}(x)$ on
    \begin{equation}
        \Omega = (-\infty, a_-] \cup [a_+, b_-] \cup [b_+, \infty),
    \end{equation}
    where $(a_-, a_+)$ and $(b_-, b_+)$ are eigengaps.
        
    \item The error curve $e(x):=S_{ab}(x)-R_{ab}(x)$ equioscillates on
    $\Omega$ with the maximal error
    \begin{equation}
	    \delta_0 := \max_{x \in \Omega} \abs{e(x)} = \min_{ g \in
	    \calR_{(2r)^2, (2r)^2}} \norm{ S_{ab}(x) - g(x) }_{ L^\infty
	    (\Omega) }
    \end{equation}
	and
    \begin{equation}
	    \frac{2}{ \rho^{(2r)^2} + 1 } \leq \delta_0 \leq
	    \frac{2}{ \rho^{(2r)^2} - 1 }, \quad \rho = \rho( \ell_1 ) > 1,
    \end{equation}
    where
    \begin{equation*}
	    \rho( \ell_1 ) = \exp \left( \frac{ \pi K(\mu') }{ 4K(\mu)}
        \right),
    \end{equation*}
    $\mu = \frac{ 1 - \sqrt{\ell_1} }{ 1 + \sqrt{\ell_1} }$, $\mu' =
    \sqrt{ 1 - \mu^2 }$, and
	$K(\mu) = \int_0^{\pi/2} \frac{ \diff \theta }{ \sqrt{ 1 -
	\mu^2 \sin^2 \theta } }$ is the complete elliptic integral of
	the first kind for the modulus $\mu$.
    \end{enumerate}

\end{theorem}

\begin{proof}

Note that inserting a rational transformation of type $(1,1)$
into a rational function of type $((2r)^2-1,(2r)^2)$ results in a
rational function of type $((2r)^2,(2r)^2)$. Since $S(x;\ell_1)\in
\calR_{(2r)^2-1,(2r)^2}$ and M\"obius transform $T(x) \in \calR_{1,1}$,
we know $R_{ab}(x) \in \calR_{(2r)^2,(2r)^2}$. In the following proof,
we will first show that $R_{ab}(x)$ is the best uniform rational
approximant of type $((2r)^2,(2r)^2)$ to the rectangular function on
$\Omega$ and then derive the error estimator.

Suppose $R_{ab}(x)$  is not the best uniform rational approximant of
type $((2r)^2, (2r)^2)$ of the rectangular function $S_{ab}(x)$ on
\begin{equation}
    \Omega = (-\infty, a_-] \cup [a_+, b_-] \cup [b_+, \infty),
\end{equation}
then there exists another rational function $\tilde{R}(x)$ in
$\calR_{(2r)^2,(2r)^2}$ such that
\[
    \norm{ S_{ab}(x) - \tilde{R}(x) }_{ L^\infty (\Omega) } <  \norm{
    S_{ab}(x) - R_{ab}(x) }_{ L^\infty (\Omega) }.
\]
Let $T^{-1}(x)$ denote the inverse transform of the M\"obius
transformation $T(x)$ in \eqref{mo}, and we have $T^{-1}\in
\calR_{1,1}$.  Note that inserting a rational transformation of
type $(1,1)$ into a rational function of type $((2r)^2,(2r)^2)$
results in a rational function of type $((2r)^2,(2r)^2)$. Hence,
$2\tilde{R}(T^{-1}(x))-1$ is a rational function approximant in
$\calR_{(2r)^2,(2r)^2}$ of the signum function $\sign{x}$ on the set
$[-1,-\ell_1] \cup [\ell_1,1]$. Note that the M\"obius transformations
$T(x)$ and $T^{-1}(x)$ are bijective maps that do not change the
approximation errors, we have
\begin{eqnarray*}
    & &\norm{ \sign x - (2\tilde{R}(T^{-1}(x))-1) }_{ L^\infty
    ([-1,-\ell_1] \cup [\ell_1,1]) } = 2 \norm{ S_{ab}(x) - \tilde{R}(x)
    }_{ L^\infty (\Omega) }\\
    & < & 2 \norm{ S_{ab}(x) - R_{ab}(x) }_{ L^\infty (\Omega) } =
    \norm{ \sign x - S(x;\ell_1)  }_{ L^\infty ([-1,-\ell_1] \cup
    [\ell_1,1]) }.
\end{eqnarray*}
The inequality
\[
    \norm{ \sign x - (2\tilde{R}(T^{-1}(x))-1) }_{ L^\infty ([-1,-\ell_1]
    \cup [\ell_1,1]) }  <  \norm{ \sign x - S(x;\ell_1)  }_{ L^\infty
    ([-1,-\ell_1] \cup [\ell_1,1]) }
\]
conflicts with the fact that $S(x;\ell_1)\in \calR_{(2r)^2-1,(2r)^2}$
is the best rational approximant (among all rational functions of
type $((2r)^2,(2r)^2)$) of the signum function on $[-1,-\ell_1] \cup
[\ell_1,1]$ by \eqref{eqn:S} and \eqref{eqn:Z}. Hence, our
previous assumption that $R_{ab}(x)$  is not the best uniform rational
approximant of type $((2r)^2, (2r)^2)$ of the rectangular function
$S_{ab}(x)$ on $\Omega$ is false. This proves the first statement of
Theorem \ref{thm:step}.

The error inequalities in Property $2)$ follow from Gon\v{c}ar's
quantitative estimation on the approximation error of Zolotarev's
functions in \cite{Goncar1969} and the bijective transformation in
\eqref{mo}.

\end{proof}

An immediate result of Theorem~\ref{thm:step} is
\begin{equation}\label{eqn:err}
    \delta_0 = \norm{ S_{ab}(x)-R(x) }_{L^\infty (\Omega)} = C_{4r^2}
    \rho^{ -4r^2 },
\end{equation}
with $1 \leq \frac{2}{ 1 + \rho^{-(2r)^2} } \leq C_{4r^2} \leq \frac{2}{
1 - \rho^{-(2r)^2} }$.

\begin{figure}[htp]
    \centering
    \begin{subfigure}[t]{0.48\textwidth}
        \centering
        \includegraphics[width=\textwidth]{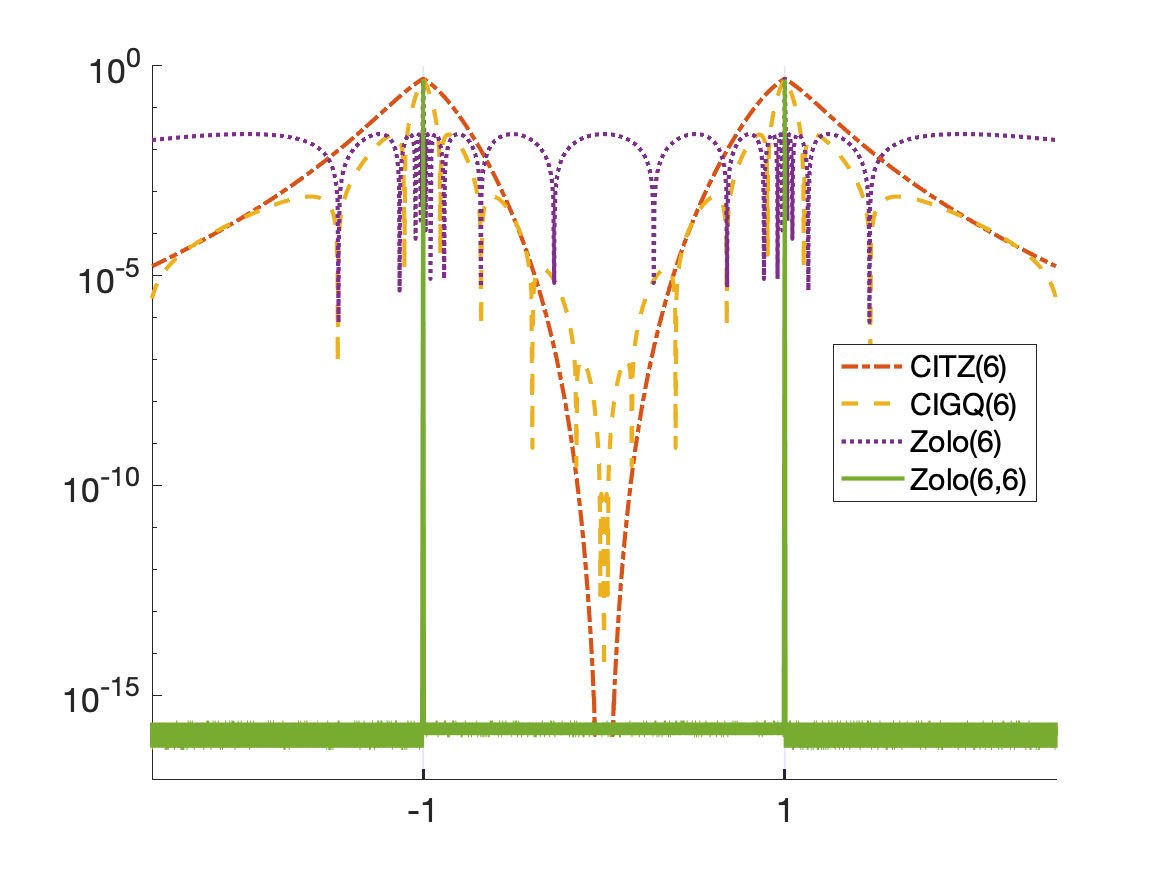}
        \caption{}
        \label{fig:ZoloErr6}
    \end{subfigure}
    \begin{subfigure}[t]{0.48\textwidth}
        \centering
        \includegraphics[width=\textwidth]{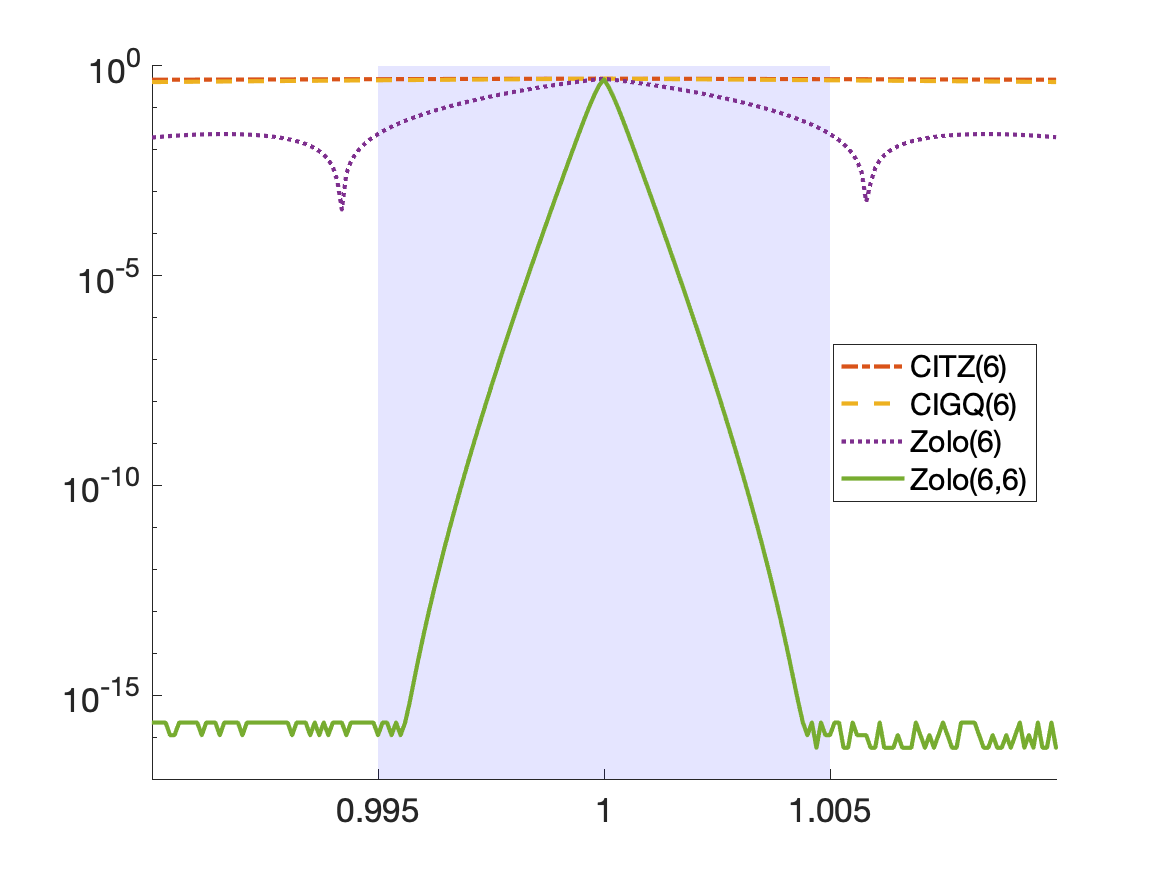}
        \caption{}
        \label{fig:ZoloErr6Zoom}
    \end{subfigure}

    \caption{ This figure shows the approximation error of various
    rational filters as an approximation of the rectangular function
    supported on $(-1,1)$. The eigengaps around $1$ and $-1$ is
    set to be $10^{-2}$. These functions include: the trapezoidal
    filter~\cite{Tang2014, Ye2016} (denoted as $\text{CITZ}(r)$,
    where $r$ is the number of poles), the Gauss-Legendre
    filter~\cite{Polizzi2009} (denoted as $\text{CIGQ}(r)$, where $r$
    is the number of poles), the Zolotarev approximation (denoted
    as $\text{Zolo}(r)$, where $r$ is the degree), and the proposed
    Zolotarev filter via compositions (denoted as $\text{Zolo}(r,r)$,
    where $r$ is the degree). (a) shows the approximation on $[-2.5,2.5]$
    and (b) zooms in on $[0.99,1.01]$.  Light purple areas are the buffer
    areas in which it is not necessary to condiser the approximation
    accuracy because of the eigengaps.}

    \label{fig:comp}
\end{figure}

\begin{figure}[htp]
    \centering
    \begin{subfigure}[t]{0.48\textwidth}
        \centering
        \includegraphics[width=\textwidth]{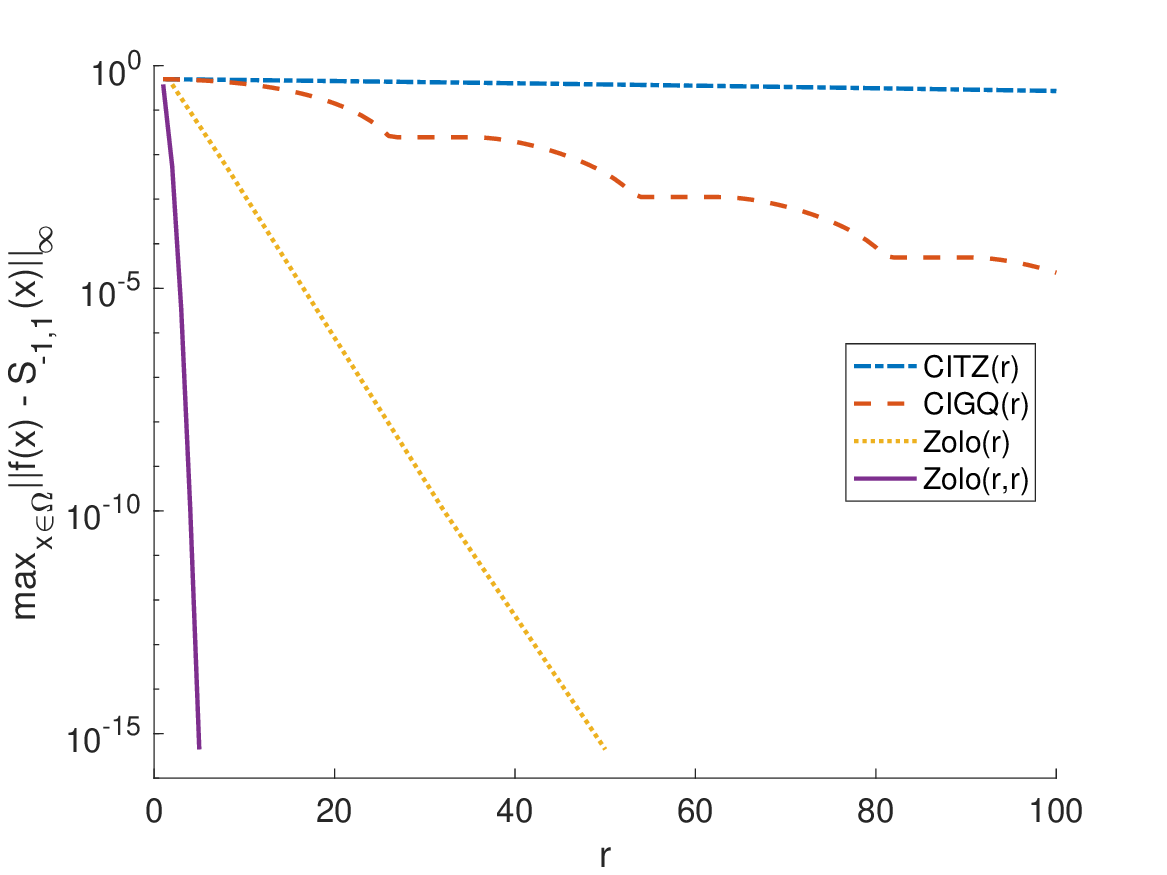}
        \caption{}
        \label{fig:ZoloErrDecay2}
    \end{subfigure}
    \begin{subfigure}[t]{0.48\textwidth}
        \centering
        \includegraphics[width=\textwidth]{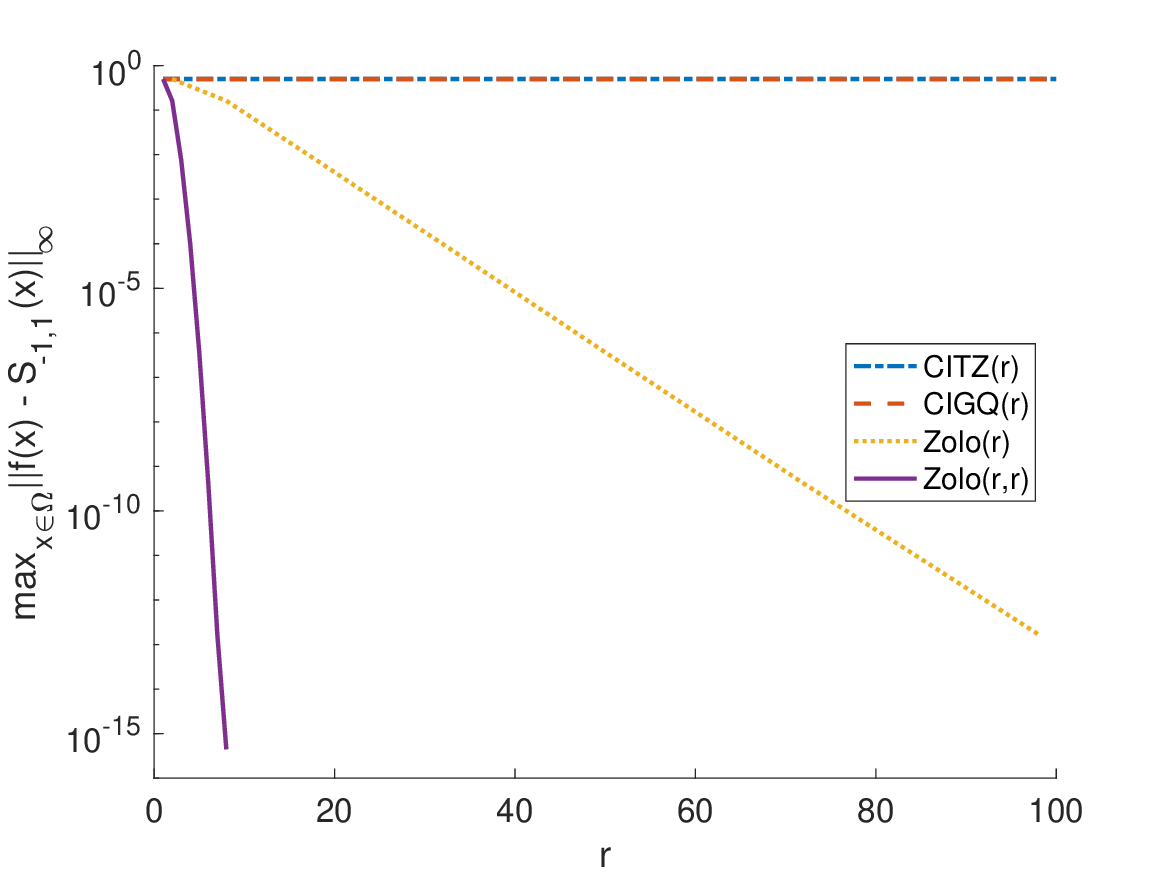}
        \caption{}
        \label{fig:ZoloErrDecay6}
    \end{subfigure}

    \caption{ This figure shows the approximation error against degree $r$
    for various rational functions as an approximation of the rectangular
    function supported on $(-1,1)$. The eigengaps around $-1$ and $1$ are
    set to be $10^{-2}$ in (a) and $10^{-6}$ in (b). The approximation
    errors of $\text{Zolo}(r,r)$ decay significantly faster than other
    methods. In (b), the line for CITZ is overwritten by that of CIGQ.}
    
    \label{fig:decay}

\end{figure}

To illustrate this improvement, we compare the performance of the
proposed rational filter in \eqref{eqn:R2} with other existing rational
filters that are constructed by discretizing the complex value contour
integral
\begin{equation}
    \pi(x) = \frac{1}{2\pi \imath} \oint_{\Gamma} \frac{1}{ x - z }\diff
    z,\quad x \notin \Gamma
\end{equation}
with an appropriate quadrature rule (e.g., the Gauss-Legendre
quadrature rule \cite{Polizzi2009} and the trapezoidal quadrature
rule \cite{Tang2014, Ye2016}). Since the dominant cost of applying
all these filters is the sparse matrix factorization, we fix the
number of matrices to be factorized and compare the approximation
error of various filters. The results in Figure \ref{fig:comp}
verifies the advantage of the proposed rational filter over existing
rational filters and shows that $6$ matrix factorizations are
enough to construct the composition of Zolotarev's rational function
approximating a rectangular function within a machine accuracy. Here
the eigengaps are $10^{-2}$. Figure~\ref{fig:decay} further explores
the decay for the errors in $L^\infty$ norm for different methods.
Figure~\ref{fig:ZoloErrDecay2} is the decay property for problem with
eigengaps $10^{-2}$ whereas Figure~\ref{fig:ZoloErrDecay6} shows the
decay property for problem with eigengaps $10^{-6}$.

\subsection{A hybrid algorithm for applying the best rational filter}
\label{sub:matvec}

In this section, we introduce a hybrid algorithm for applying the best
rational filter $R_{ab}(x)$ constructed in Section \ref{sub:high},
i.e., computing the matvec $R_{ab}(B^{-1}A) V$ when $A$ and $B$
are sparse Hermitian matrices in $\bbF^{N\times N}$ and $V$ is a
vector in $\bbF^{N}$.  Recall that the rational filter $R_{ab}(x)$
is constructed by
\begin{equation}\label{eqn:R3}
    R_{ab}(x) = \frac{ Z_{2r}( \hZ_{2r}( T(x); \ell_1 ); \ell_2 ) +
    1 }{2}.
\end{equation}
Hence, it is sufficient to show how to compute $Z_{2r}( \hZ_{2r}( T(
B^{-1}A ); \ell_1 ); \ell_2 ) V$ efficiently.

For the sake of numerical stability and parallel computing, a rational
function is usually evaluated via a partial fraction representation in
terms of a sum of fractions involving polynomials of low degree. For the
Zolotarev's function $Z_{2r}(x; \ell)$ introduced in \eqref{eq:ZoloFunc},
we have the following partial fraction representation\footnote{The existence of the partial fraction representation is well-known. We present our formulas for the representation for the purpose of making our algorithm easier to implement for researchers who are interested in our work.}. The reader is referred to 
Appendix for the proof.

\begin{proposition}\label{pro:p1}
    The function $Z_{2r}(x;\ell)$ as in \eqref{eq:ZoloFunc} can be
    reformulated as
    \begin{equation}\label{eqn:pa}
	Z_{2r}(x; \ell) = M x \frac{ \prod_{j = 1}^{r-1}( x^2 + c_{2j}
	)}{ \prod_{j = 1}^{r}( x^2 + c_{2j-1})} = M x \left( \sum_{j =
	1}^r \frac{a_j}{ x^2 + c_{2j-1} } \right),
    \end{equation}
    where
    \begin{equation}
    a_j = \frac{b_j}{ c_{2r-1} - c_{2j-1} }
    \end{equation}
    for $j = 1, \dots, r-1$, and
    \begin{equation}
    a_r = 1 - \sum_{j=1}^{r-1} \frac{b_j}{ c_{2r-1} - c_{2j-1} }.
    \end{equation}
    Here
    \begin{equation}
    b_j = (c_{2j} - c_{2j-1}) \prod_{ k = 1, k \neq j}^{r-1} \frac{
    c_{2k} - c_{2j-1} }{ c_{2k-1} - c_{2j-1} }
    \end{equation}
    for $j = 1, \dots, r-1$, $\{c_j\}$ and $M$ are given in
    \eqref{eqn:coef}.
\end{proposition}

If complex coefficients are allowed, the following corollary can be
derived from Proposition~\ref{pro:p1} directly.

\begin{corollary}\label{cor:p1}
    The function $Z_{2r}(x;\ell)$ as in \eqref{eq:ZoloFunc} can be
    reformulated as
    \begin{equation}\label{eqn:paRe}
	Z_{2r}(x; \ell) = \frac{M}{2} \sum_{j = 1}^r \left(
	\frac{a_j}{ x + \imath \sqrt{c_{2j-1}} } + \frac{a_j}{ x -
	\imath \sqrt{c_{2j-1}} } \right),
    \end{equation}
    where $a_j$ and $c_{2j-1}$ are as defined in
    Proposition~\ref{pro:p1}.
\end{corollary}

By Proposition~\ref{pro:p1}, we obtain the partial fraction
representation of $Z_{2r}(T(x);\ell)$ as follows, where $T(x)$ is a
M\"obius transformation $T(x)=\gamma\frac{x-\alpha}{x-\beta}$. The reader is referred to 
Appendix for the proof.

\begin{proposition}\label{pro:p2}
The function $Z_{2r}(T(x);\ell)$ can be reformulated as
\begin{equation}
    Z_{2r}( T(x); \ell) = M \sum_{j = 1}^r \frac{a_j \gamma}{\gamma^2
    + c_{2j-1}} + M \sum_{j = 1}^r \left( \frac{w_j}{x - \sigma_j} +
    \frac{\bar{w}_j}{x - \bar{\sigma}_j}\right).
\end{equation}
where
\begin{equation}
    \sigma_j = \frac{\gamma \alpha + \imath \sqrt{c_{2j-1}} \beta
    }{\gamma + \imath \sqrt{c_{2j-1}} }, \quad w_j = \frac{a_j (\sigma_j
    - \beta)}{2(\gamma + \imath
	\sqrt{c_{2j-1}})}.
\end{equation}
\end{proposition}

\begin{remark}
    In the rest of this paper, we denote the constants associated with
    $Z_{2r}(x; \ell_2)$ as $a_j, c_{2j-1}, \sigma_j,$ and $w_j$ for $j =
    1, \dots, r$; and the constants associated with $\hZ_{2r}(x; \ell_1)$
    as $\ha_j, \hc_{2j-1}, \hsigma_j,$ and $\hw_j$ for $j = 1, \dots, r$.
\end{remark}

\begin{figure}[htb]
    \centering
    \begin{subfigure}[t]{0.48\textwidth}
        \centering
        \includegraphics[width=\textwidth]{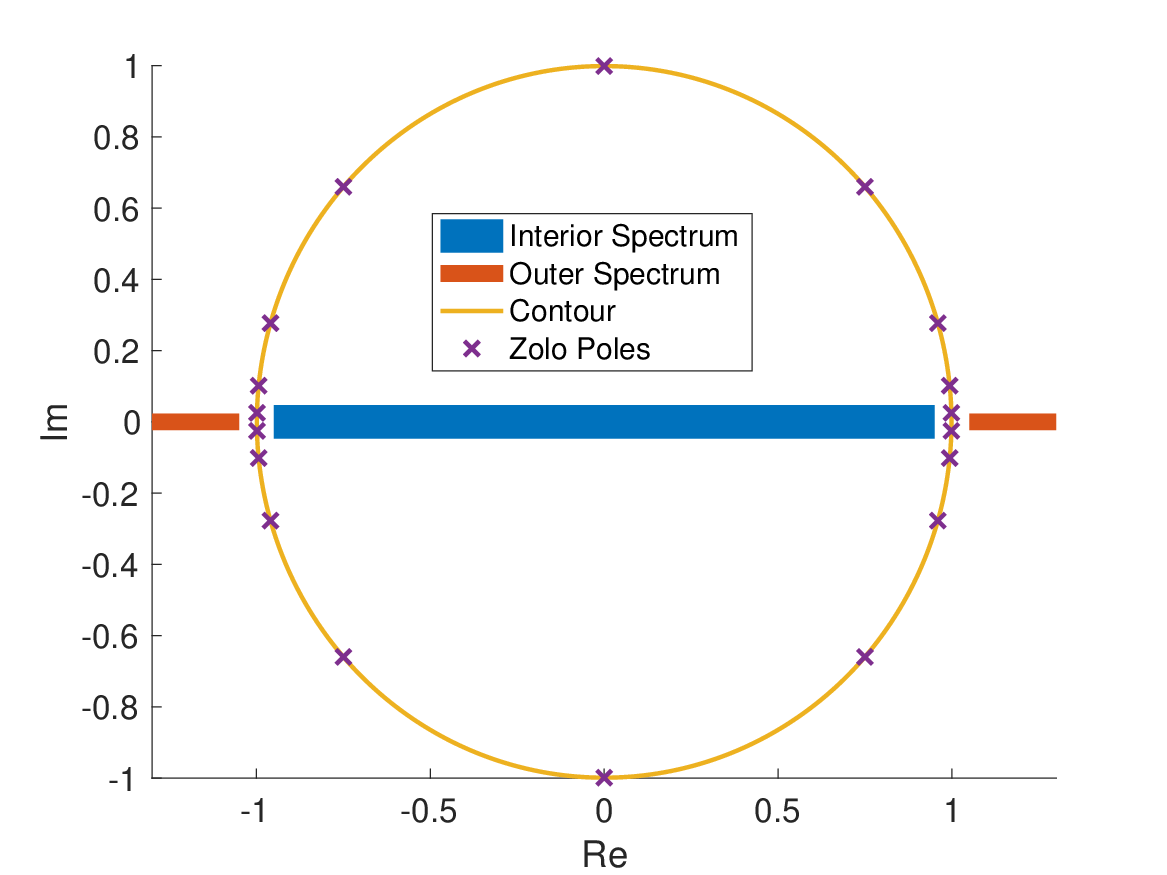}
        \caption{}
        \label{fig:ZoloPole9}
    \end{subfigure}
    \begin{subfigure}[t]{0.48\textwidth}
        \centering
        \includegraphics[width=\textwidth]{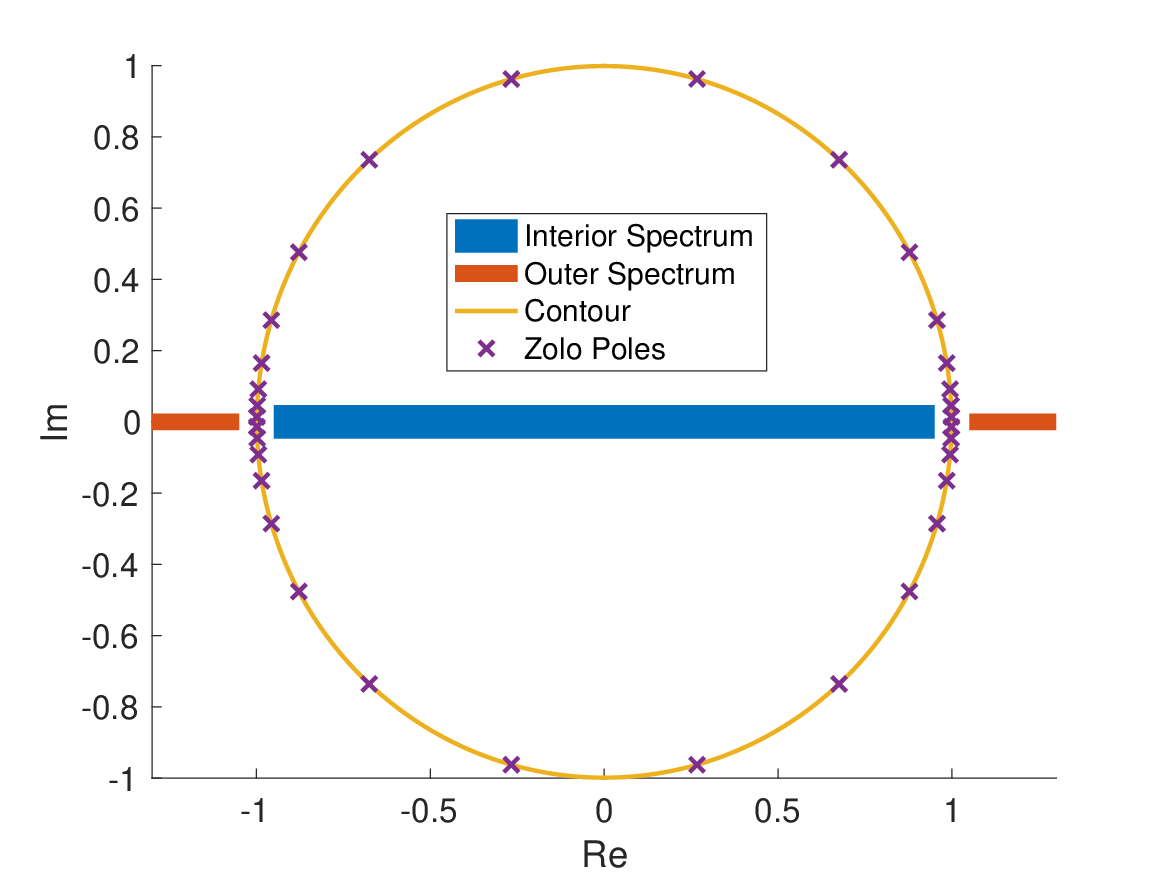}
        \caption{}
        \label{fig:ZoloPole16}
    \end{subfigure}

    \caption{The corresponding contour integral discretization of
    Zolotarev's function composed with M\"obius transformation. The
    eigengaps here are $(-1.1,-0.9)$ and $(0.9,1.1)$ and the contour is a
    circle centered at origin with radius $0.998749$. The discretization
    points are calculated as Proposition~\ref{pro:p2}. (a) $r = 9$
    provides 18 poles; (b) $r = 16$ provides 32 poles.}

    \label{fig:ZoloPole}

\end{figure} 

Proposition~\ref{pro:p2} can be viewed as a discretization of a contour
at poles, $\sigma_j$ and $\hsigma_j$ with weights $w_j$ and $\hw_j$ for
$j = 1, \dots, r$. The contour is a circle centered on the real axis
cutting through the eigengaps. Figure~\ref{fig:ZoloPole} demonstrate
an example with eigengaps $(-1.1,-0.9)$ and $(0.9,1.1)$. The calculated
contour is a circle centered at origin with radius $0.998749$. Meanwhile,
the pole locations and the corresponding weights are provided by
Proposition~\ref{pro:p2}.  Figure~\ref{fig:ZoloPole9} adopts $r = 9$
which is also the composition of two Zolotarev's functions with degree
$3$ whereas Figure~\ref{fig:ZoloPole16} adopts $r = 16$ which is also
the composition of two Zolotarev's functions with degree $4$.

With these propositions ready, we now introduce the hybrid algorithm for
computing the matvec $R_{ab}(B^{-1}A)V = Z_{2r}(\hZ_{2r}(T(B^{-1}A);
\ell_1); \ell_2)V$. This hybrid algorithm consists of two parts of
linear system solvers: an inner part and an outer part. The inner part
implicitly computes the matvec $\hZ_{2r}(T(B^{-1}A); \ell_1)V := GV$ via
fast direct solvers. Once $GV$ has been implicitly computed, the matrix
$G$ can be viewed as an operator with fast application algorithm, where
each application costs nearly $O(N)$ operations in many applications.
The outer part computes $Z_{2r}(G;\ell_2)V$ using a GMRES method when
the fast matvec $GV$ is available. Since the matrix $G$ has singular
values greater than $\ell_2 = \hZ_{2r}(\ell_1; \ell_1)$, which is
a number close to 1, a few steps of iterations in GMRES method are
enough to solve the linear systems in the matvec $Z_{2r}(G;\ell_2)V$
accurately. In practice, the iteration number varies from 6 to 25.

In particular, by Proposition \ref{pro:p2},
\begin{equation}
    \begin{split}
        GV = & \hZ_{2r}(T(B^{-1}A); \ell_1) V \\
           = & \hM \sum_{j=1}^r \frac{\ha_j\gamma }{\gamma^2+\hc_{2j-1}} V 
            + \hM \sum_{j=1}^r \left( \left( \hw_j \left( A - \hsigma_j
            B \right)^{-1} BV \right) + \left( \overline{\hw}_j \left(
            A - \overline{\hsigma}_j B \right)^{-1} BV \right) \right) \\
           = & \hM \sum_{j=1}^r \frac{\ha_j\gamma }{\gamma^2+\hc_{2j-1}} V 
            + \hM \sum_{j=1}^r \left( \left( \hw_j \left( A - \hsigma_j
            B \right)^{-1} BV \right) + \left( \overline{\hw}_j \left(
            A - \hsigma_j B \right)^{-*} BV \right) \right).
        \label{eqn:solve}
    \end{split}
\end{equation}
The third equality holds since $A$ and $B$ are Hermitian matrices.
Hence, evaluating $\hZ_{2r}(T(B^{-1}A);\ell_1)V$ boils down to solving
$r$ linear systems of the form
\begin{equation}\label{eqn:mts}
(A-\hsigma_j B)x=y
\end{equation}
for $j=1,\dots,r$. This is a set of $r$ sparse linear systems. Since
the operator $G$ is involved in an outer function, where it is
repeatedly applied, a fast and efficient algorithm for applying $G$
is necessary. This can also be rephrased as ``a fast and efficient
algorithm for solving \eqref{eqn:mts} is necessary''. There are two
groups of efficient algorithms for solving \eqref{eqn:mts}: direct
solvers and iterative solvers with efficient preconditioners.

Fast direct solvers for sparse linear system as $A-\hsigma_j B$ usually
contains two phases. The first phase (termed as the pre-factorization
phase) factorizes the sparse matrix into a product of a sequence
of lower and upper triangular sparse matrices. The second phase
(termed as the solving phase) solves the sequence of triangular sparse
matrices efficiently against vectors. The computational complexities
for both the pre-factorization and the solving phase vary from
method to method, also heavily rely on the sparsity pattern of the
matrix. For simplicity, we denoted the computational complexity for the
pre-factorization and the solving phase as $F_N$ and $S_N$ respectively
for matrices of size $N \times N$. Usually, $F_N$ is of higher order
in $N$ than $S_N$. Particularly, we adopt the multifrontal method
(MF)~\cite{Duff1983,Liu1992} as the general direct sparse solver for all
numerical examples in this paper. For sparse matrices of size $N\times
N$ from two-dimensional PDEs, the computational complexities for MF are
$F_N = O(N^{3/2})$ and $S_N = O(N\log N)$. While, for three-dimensional
problems, MF requires $F_N = O(N^2)$ and $S_N=O(N^{4/3})$ operations.

Iterative solvers with efficient preconditioners is another efficient
way to solve sparse linear systems. The construction of preconditioners
is the pre-computation phase whereas the iteration together with applying
the preconditioner is the solving phase. Similarly to the direct solver,
the choices of iterative solvers and preconditioners highly depend on
sparse matrices. For elliptic PDEs, GMRES could be used as the iterative
solver for $A-\hsigma_j B$, and MF with reduced frontals~\cite{Xia2013,
Schmitz2014, Ho2015, Li2016_DHIF} could provide good preconditioners.

Once the fast application of $G$ is available, we apply the
classical GMRES together
with the shift-invariant property of the Krylov subspace (See
\cite{Saad2003} Section 7.3)  to evaluate $R_{ab}(B^{-1}A)V =
Z_{2r}(\hZ_{2r}(T(B^{-1}A);\ell_1);\ell_2)V = Z_{2r}(G;\ell_2)V$.
In more particular, by Corollary~\ref{cor:p1}, we have
\begin{equation}
    Z_{2r}(G;\ell_2)V = M \sum_{j=1}^r \frac{a_j}{2} \left(
    \left( G + \imath
    \sqrt{c_{2j-1}}I \right)^{-1} V + 
    \left( G - \imath
    \sqrt{c_{2j-1}}I \right)^{-1} V
    \right),
    \label{eqn:solve2}
\end{equation}
where $I$ is the identify matrix.  Hence, to evaluate
$Z_{2r}(G;\ell_2)V$, we need to solve multi-shift linear systems of
the form
\begin{equation}\label{eqn:bcg}
    (G \pm \imath \sqrt{c_{2j-1}}I)x=y
\end{equation}
with $2r$ shifts $\pm \imath \sqrt{c_{2j-1}}$ for $j = 1, \dots,
r$. These systems are solved by the multi-shift GMRES method
efficiently. In each iteration, only a single evaluation of $GV$
is needed for all shifts. Meanwhile, since $G$ has a condition
number close to $1$, only a few iterations are sufficient to solve the
multi-shift systems to a high accuracy. Let the number of columns in $V$
be $O(\nlam)$ and the number of iterations be $m$.  The complexity for
evaluating the rational filter $R_{ab}(B^{-1}A)V$ is $O(m \nlam S_N)$.

\begin{algorithm2e}[H]

\DontPrintSemicolon
\SetAlgoNoLine
\SetKwInOut{Input}{input}
\SetKwInOut{Output}{output}

    \Input{A sparse Hermitian definite matrix pencil $(A, B)$, a spectrum
    range $(a,b)$, vectors $V$, tolerance $\eps$ }

    \Output{$R_{ab}(B^{-1}A)V$ as defined in \eqref{eqn:R3}}

    Estimate eigengaps $(a_-,a_+)$ and $(b_-,b_+)$ for $a$ and $b$
    respectively.

    Solve \eqref{eqn:req} for $\ell_1$ and M\"obius transformation parameter
    $\gamma, \alpha, \beta$.

    Given $\eps$, find the smallest order of Zolotarev's functions, $r$,
    such that our rational function approximates $S_{ab}$ within the target
    accuracy $\eps$ by gradually increasing $r$.

    Calculate function coefficients, $\hM, \ha_j, \hw_j, \hsigma_j,
    \hc_{2j-1}$ and $\ell_2, M, a_j, c_{2j-1}$ for $j = 1, \dots, r$.

    \For{ $j = 1, 2, \dots, r$}{ Pre-factorize $A - \hsigma_j B$ as
    $K_j$ }

    Generate algorithm for operator $$GV = \hM \sum_{j=1}^r
    \frac{\ha_j\gamma }{\gamma^2+\hc_{2j-1}} V + \hM \sum_{j=1}^r \left(
    \hw_j K_j^{-1} BV + \overline{\hw}_j K_j^{-*} BV\right).$$

    Apply the multi-shift GMRES method for solving linear systems $(G
    \pm \imath \sqrt{c_{2j-1}}I)^{-1}V$ with $j = 1, \dots, r$.

    $R_{ab}(B^{-1}A)V = \frac{M}{2} \sum_{j=1}^r \frac{a_j}{2} \left(
    \left( G + \imath \sqrt{c_{2j-1}}I \right)^{-1} V + \left( G -
    \imath \sqrt{c_{2j-1}}I \right)^{-1} V  \right) + \frac{1}{2}V$

    \caption{A hybrid algorithm for the rational filter
    $R_{ab}(B^{-1}A)$} \label{alg:hybrid}

\end{algorithm2e}

Algorithm~\ref{alg:hybrid} summarizes the hybrid algorithm
introduced above for applying the rational filter $R_{ab}(B^{-1}A)$
in \eqref{eqn:R3} to given vectors $V$. By taking Line 1-8 in
Algorithm~\ref{alg:hybrid} as precomputation and inserting Line 9-10 in
Algorithm~\ref{alg:hybrid} into Line 3 in Algorithm~\ref{alg:subiter},
we obtain a complete algorithm for solving the interior generalized
eigenvalue problem on a given interval $(a,b)$. When the matrix pencil
$(A,B)$ consists of sparse complex Hermitian definite matrices, the
dominant cost of the algorithm is the pre-factorization of $r$ matrices
in \eqref{eqn:mts} or Line 6 in Algorithm~\ref{alg:hybrid}.

\begin{remark}
    Given a desired accuracy $\epsilon$ and the parameter $\ell_1$ computed
    from the estimated eigengaps, we can estimate the order of Zolotarev's
    functions efficiently, which corresponds to the third line in
    Algorithm~\ref{alg:hybrid}. Notice that the $L^\infty$ error of
    $S(x;\ell_1)$ as in \eqref{eqn:S} approximating the signum function is
    achieved at $x = \ell_1$. Therefore, in practice, we evaluate $S(\ell_1;
    \ell_1)$ for a sequence of $r$s and choose the smallest $r$ such that
    the error is bounded by $\epsilon$. Since the evaluation of $S(\ell_1;
    \ell_1)$ does not involve any matrix, the estimation of the order $r$
    can be done efficiently. If different $r_1$ and $r_2$ are of interest, a
    small table of $S(\ell_1; \ell_1)$ can be computed as a reference for
    uses to select a pair of $(r_1,r_2)$ from it.
\end{remark}

\section{Numerical examples}\label{sec:results}

In this section, we will illustrate three examples based on different
collections of sparse matrices. The first example aims to show the scaling
of the proposed method; the second example shows the comparison
with the state of the art algorithm for spectrum slicing problem,
FEAST~\cite{Polizzi2009, Guttel2015}; the last example shows the efficiency of the proposed method for various kinds of sparse matrices. All numerical examples are performed
on a desktop with Intel Core i7-3770K 3.5 GHz, 32 GB of memory. The
proposed algorithm in this paper is implemented in MATLAB R2017b, which is
shorten as ``ZoloEig'' or ``Zolo'' in this section. And the FEAST v3.0
compiled with Intel compiler produces the results in the part of
``FEAST''. To make the numerical results reproducible, the codes for the numerical 
examples can be found in the authors' personal homepages.

Throughout the numerical section, a relative error without knowing the underlying ground true eigenpairs is used to measure the accuracy of both ZoloEig and FEAST. The relative error of the estimated interior eigenpairs
in the interval $(a,b)$ is defined as
\begin{equation}\label{eq:relerr}
    e_{\Lambda,X} = \max_{1\leq i \leq k} \frac{ \norm{ A X_i - B X_i \lambda_i}_2 }{
        \norm{ \max (\abs{a},\abs{b}) B X_i }_2},
\end{equation}
where $(A,B)$ is the matrix pencil of size $N$ by $N$; $\Lambda \in
\bbR^{k \times k}$ is a diagonal matrix with diagonal entries being the
estimated eigenvalues in the given interval, $\{\lambda_1, \lambda_2, \dots,
\lambda_k\}$; and $X_i \in \bbF^{N \times 1}$ denotes the $i$-th
eigenvector for $1\leq i \leq k$. This relative error of the eigenvalue
decomposition is also used in ZoloEig as the stopping criteria.  Besides
the error measurement, we also define a measurement of the difficulty of
the problem as the relative eigengap,
\begin{equation}
    \delta_\lambda = \frac{\min(a_+ - a_-, b_+ - b_-)}{b_- - a_+}.
\end{equation}
Such a relative eigengap, $\delta_\lambda$ can measure the intrinsic
difficulty of the spectrum slicing problem for all existing algorithms based on polynomial filters and rational function filters.

\begin{table}[htp]
    \centering
    \begin{tabular}{cl}
        \toprule
        Notation & Description \\
        \toprule
        $n_{ss}$ & Size of subspace used in ZoloEig or FEAST. \\
        $n_{iter}$ & Number of subspace iterations used by ZoloEig or FEAST. \\
        $n_{gmres}$ & Number of GMRES iterations used by ZoloEig. \\
        $n_{solv}$ & Total number of linear system solves used by ZoloEig
        or FEAST. \\
        $T_{fact}$ & Total factorization time used by ZoloEig in second. \\
        $T_{iter}$ & Total iteration time used by ZoloEig in second. \\
        $T_{total}$ & Total runtime used by ZoloEig in second. \\
        \bottomrule
    \end{tabular}
    \caption{Notations used in numerical results.}
    \label{tab:numnotations}
\end{table}

Other notations are listed in Table~\ref{tab:numnotations}. The total number
of linear system solves in ZoloEig can be calculated as,
\begin{equation}
    n_{solv} = r \cdot n_{ss} \cdot n_{iter} \cdot n_{gmres},
\end{equation}
whereas the one in FEAST is
\begin{equation}
    n_{solv} = r \cdot n_{ss} \cdot n_{iter}.
\end{equation}

\subsection{Spectrum of Hamiltonian Operators}

The first example is a three-dimensional Hamiltonian operator,
\begin{equation}\label{eq:ex-Ham}
    H = -\frac{1}{2} \Delta + V,
\end{equation}
on $[0,1)^3$ with a Dirichlet boundary condition. Here $V$ is a
three-dimensional potential field containing three Gaussian wells with
random depths uniformly chosen from $(0,1]$ and fixed radius 0.2.
Figure~\ref{fig:gaussianwell} shows the isosurface of an instance of the
3D random Gaussian well. This example serves the role of illustrating the
efficiency and complexity of the proposed new algorithm. We first
discretize the domain $[0,1)^3$ by a uniform grid with $n$ points on each
dimension and the operator is discretized with 7-point stencil finite
difference method that results in a sparse matrix. The multifrontal method
is naturally designed for inverting such sparse matrices. In this section,
we adopt Matlab ``eigs'' function to evaluate the smallest 88 eigenpairs
as the reference. Due to the randomness in the potential, the relative
eigengap of the smallest 88 eigenvalues varies a lot. In order to obtain
the scaling of the algorithm, we prefer to have problems of different
sizes but with similar difficulty. Therefore, we generate random potential
fields until the problem has a relative eigengap between $10^{-3}$ and
$10^{-4}$.  In such cases, the claimed complexity of the ZoloEig algorithm
can be rigorously verified for the discretized operator of
\eqref{eq:ex-Ham}.  In this example, the tolerance is set to be $10^{-8}$,
$r$ is set as $(4,4)$ for all matrices, and subspaces with dimension 89
are used to recover the 88 eigenpairs.

\begin{figure}[htp]
\begin{center}
    \includegraphics[height=3in]{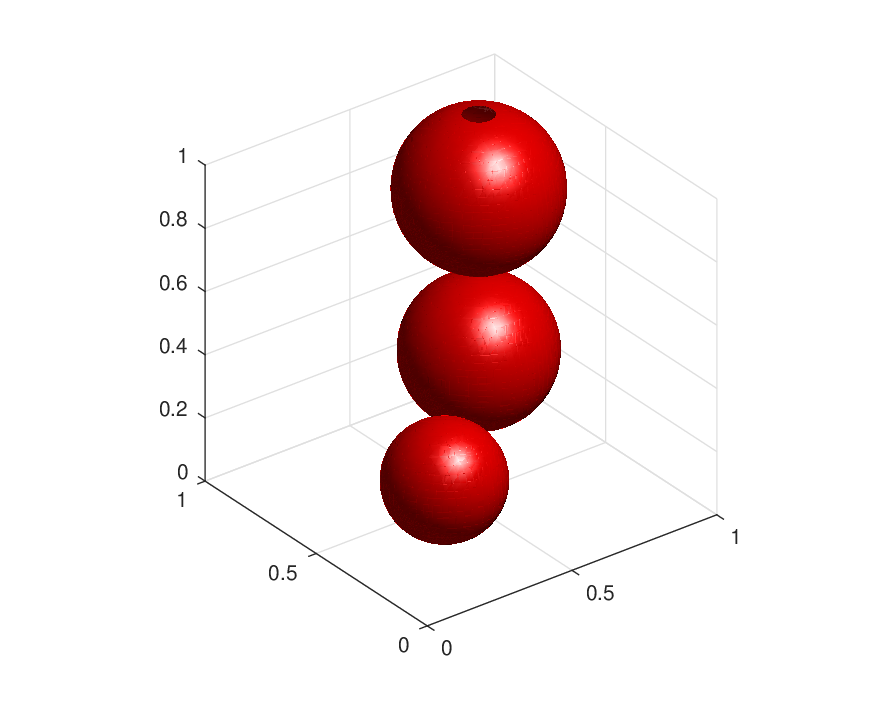}
\end{center}
\caption{An instance of 3D random potential field using Gaussian wells. The
    isosurface is at level $-0.5$.} \label{fig:gaussianwell}
\end{figure} 

\begin{figure}[htp]
    \centering
    \begin{subfigure}[t]{0.48\textwidth}
        \centering
        \includegraphics[width=\textwidth]{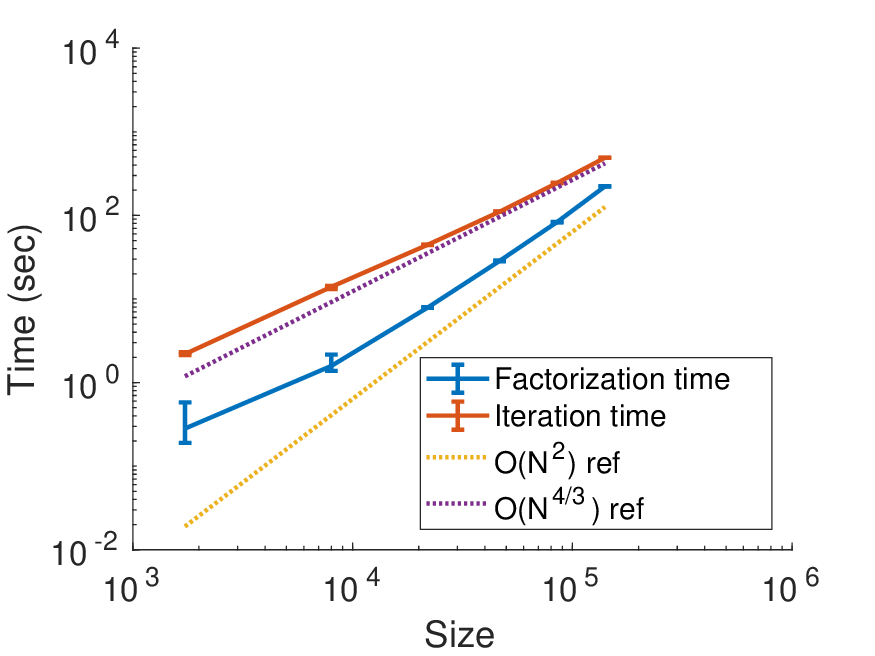}
        \caption{}
        \label{fig:ex1-time}
    \end{subfigure}
    \begin{subfigure}[t]{0.48\textwidth}
        \centering
        \includegraphics[width=\textwidth]{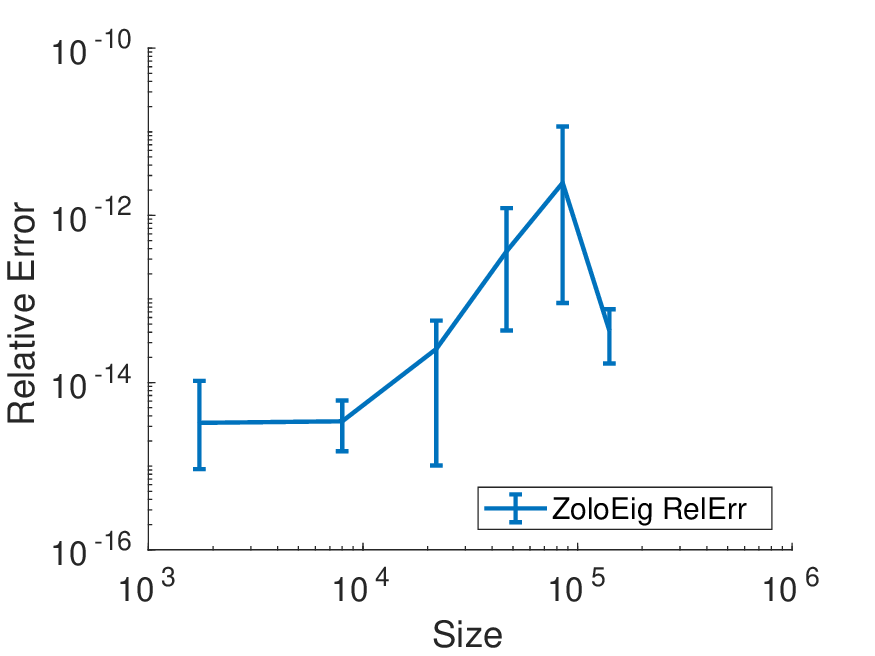}
        \caption{}
        \label{fig:ex1-err}
    \end{subfigure}

    \caption{The running time and the relative error for 3D Hamiltonian
    operator with Gaussian wells solved via ZoloEig. The relative error
    here is the relative error of eigenpairs defined in \eqref{eq:relerr}.}
    \label{fig:ex1}
\end{figure} 

\begin{table}[htp]
    \centering
    \scriptsize
    \begin{tabular}{ccccccccccc}
        \toprule
        $N$ & $\delta_\lambda$ & $r$ & $e_{\Lambda,X}$
        & $n_{ss}$ & $n_{iter}$ & $n_{gmres}$ & $n_{solv}$
        & $T_{fact}$ & $T_{iter}$ & $T_{total}$ \\
        \toprule
          1728 & 8.6e-04 & (4,4) & 3.3e-15 &  89 &   1 &  14 & 4984
        & 2.8e-01 & 2.2e+00 & 2.6e+00\\
          8000 & 5.1e-04 & (4,4) & 3.4e-15 &  89 &   1 &  16 & 5625
        & 1.6e+00 & 1.4e+01 & 1.6e+01\\
         21952 & 4.3e-04 & (4,4) & 2.5e-14 &  89 &   1 &  15 & 5340
        & 7.9e+00 & 4.5e+01 & 5.2e+01\\
         46656 & 6.6e-04 & (4,4) & 3.7e-13 &  89 &   1 &  15 & 5340
        & 2.9e+01 & 1.1e+02 & 1.4e+02\\
         85184 & 2.4e-04 & (4,4) & 2.5e-12 &  89 &   1 &  16 & 5696
        & 8.3e+01 & 2.4e+02 & 3.3e+02\\
        140608 & 1.6e-04 & (4,4) & 4.2e-14 &  89 &   1 &  17 & 6052
        & 2.2e+02 & 4.9e+02 & 7.1e+02\\
        \bottomrule
    \end{tabular}

    \caption{Numerical results for 3D Hamiltonian Operators. $N$ is the
    size of the sparse matrix, $r$ is the order used in ZoloEig, other
    notations are as defined in Table~\ref{tab:numnotations}.}
    
    \label{tab:ex1}
\end{table}

Figure~\ref{fig:ex1} shows the running time and the relative error of
eigenvalues, $e_{\Lambda,X}$. The 3D problem size varies from $12^3$ to
$52^3$ and the corresponding matrix size varies from $1,728$ to $140,608$.
The order $r$ is $(4,4)$ in the ZoloEig.  For each matrix, we provide the
true eigenvalues $\lambda_1, \lambda_{88}, \lambda_{89}$ as the input,
$a_-=-\infty, a_+=\lambda_1, b_-=\lambda_{88}, b_+=\lambda_{89}$, where
$\lambda_1$ is the smallest eigenvalue, $\lambda_{88}$ and $\lambda_{89}$
are the 88th and 89th small eigenvalues. The ZoloEig is executed 5 times
with different initial random vectors for each matrix. In
Figure~\ref{fig:ex1-time} and Figure~\ref{fig:ex1-err}, these results are
presented in a bar plot manner: the vertical bars indicate the largest and
the smallest values, whereas the trend line goes through the mean values.
And Table~\ref{tab:ex1} shows the means of the results across 5 runs.  As
we can read from Figure~\ref{fig:ex1-time}, the iteration time for ZoloEig
scales as $N^{4/3}$ while the factorization time scales as $N^2$. Both of
them agree with the scaling of multifrontal method.  Although for the
examples here, the iteration time is more expensive than the factorization
time, as $N$ getting larger, the total runtime will quickly be dominated
by the factorization. Therefore, reducing the number of factorizations
would significantly reduce the cost of the algorithm.
Figure~\ref{fig:ex1-err} shows the relative error of the eigenvalues,
which in general increases mildly as the problem size increases. All the
relative errors are achieved with only one subspace iteration. At the same
time, we find that the errors are far smaller than the tolerance
$10^{-8}$. This implies that setting $r=(4,4)$ overkills the problem and, in
practice, user could use smaller $r$.

\subsection{Hamiltonian of Silicon Bulk}

The second example is a sparse Hermitian definite matrix pencil, $(A,B)$,
generated by SIESTA (a quantum chemistry software). For a silicon bulk in
3D with $y^3$ supercell of cubic Si, a DZP basis set with radius 4
\si{\angstrom} is adopted to discretize the system, where $y=2, 3, 4, 5$.
In the spectrum slicing problem, the interval is chosen to contain the
smallest 93 eigenvalues. The ZoloEig algorithm with $r = (3,3)$, $n_{ss} =
94$ is used to solve the eigenvalue problem. The tolerance for both
ZoloEig and FEAST is set to be $10^{-14}$. Here we discuss the choice of
the parameters used in FEAST, as in Table~\ref{tab:ex2-comp}, in detail.
Since we want to keep the number of factorizations as small as possible,
we test FEAST with fixed $n_{ss}=200$ and gradually increasing $r$
starting from 3 until the first $r$ that FEAST converges. Later, given the
$r$, we choose $n_{ss}$ that minimize $n_{solv}$. Therefore, we have tried 
our best to obtain the optimal parameters for FEAST to maintain the 
smallest possible number of factorizations.

\begin{table}[htp]
    \centering
    \tiny
    \begin{tabular}{cccccccccccc}
        \toprule
        $y$ & $N$ & $\delta_\lambda$ & $r$ & $e_{\Lambda,X}$
        & $n_{ss}$ & $n_{iter}$ & $n_{gmres}$ & $n_{solv}$ & $T_{fact}$
        & $T_{iter}$ & $T_{total}$ \\
        \toprule
        2 &   832 & 9.6e-02 & (3,3) & 4.7e-17 & 94 & 1 & 10 & 2820 & 1.1e+00
          & 2.4e+00 & 3.5e+00 \\
        3 &  2808 & 3.5e-01 & (3,3) & 9.9e-16 & 94 & 1 &  7 & 1974 & 9.4e+00 
          & 7.9e+00 & 1.7e+01 \\
        4 &  6656 & 3.3e-02 & (3,3) & 3.8e-15 & 94 & 1 & 11 & 3102 & 4.4e+01
          & 4.2e+01 & 8.6e+01 \\
        5 & 13000 & 9.0e-02 & (3,3) & 5.9e-15 & 94 & 1 &  9 & 2538 & 1.6e+02
          & 8.0e+01 & 2.4e+02 \\
        \bottomrule
    \end{tabular}

    \caption{Numerical results of ZoloEig for generalized eigenvalue problems from SIESTA. $y$
    is the number of unit cell on each dimension and other notations are as
    in Table~\ref{tab:ex1}.}
    
    \label{tab:ex2}
\end{table}

\begin{table}[htp]
    \centering
    \begin{tabular}{c|cccccc|ccccc}
        \toprule
	     & \multicolumn{6}{c|}{ZoloEig} & \multicolumn{5}{c}{FEAST} \\
	    $y$ & $r$ & $e_{\Lambda,X}$ & $n_{ss}$ & $n_{iter}$ & $n_{gmres}$ &
	    $n_{solv}$ & $r$ & $e_{\Lambda,X}$ & $n_{ss}$ & $n_{iter}$ &
	    $n_{solv}$ \\

        \toprule
        2 & (3,3) &4.7e-17 & 94 & 1 & 10 & 2820 & 3 & 8.3e-15 & 97 & 29 & 8439 \\
        3 & (3,3) &9.9e-16 & 94 & 1 &  7 & 1974 & 4 & 3.4e-15 & 94 & 19 & 7144 \\
        4 & (3,3) &3.8e-15 & 94 & 1 & 11 & 3102 & 6 & 8.2e-15 & 96 & 11 & 6336 \\
        5 & (3,3) &5.9e-15 & 94 & 1 &  9 & 2538 & 7 & 9.4e-15 & 112 & 9 & 7056 \\
        \bottomrule
    \end{tabular}

    \caption{Comparison between ZoloEig and FEAST in the example of SIESTA. Notations are as in
    Table~\ref{tab:ex1}}
    
    \label{tab:ex2-comp}
\end{table}

Table~\ref{tab:ex2} includes the detail information of the numerical
results of ZoloEig. According to column $T_{fact}$ and $T_{iter}$, we find
the same scaling as in the first example. However, the factorization time
is more expensive here due to the increase of the non-zeros in the
Hamiltonian. And the total time is dominated by the factorization when
$N=13000$. Therefore, it is worth to emphasize again that reducing the
number of factorizations is important.

Table~\ref{tab:ex2-comp} provides the comparison between ZoloEig and FEAST
in the sequential cases.  Note that these two algorithms were implemented in
different programming languages: ZoloEig is implemented in MATLAB and FEAST
is in Fortran. Direct comparison of the runtime is unfair for ZoloEig, since
MATLAB code is usually about 5x to 10x slower than Fortran
code~\footnote{Even though there is difference between programming
languages, we find that the actual runtime of ZoloEig is still faster than
that of FEAST for large problem sizes, namely when $y\geq 4$ in Table
\ref{tab:ex2-comp}.}. Hence, we compare the total number of linear system
solves here, which is the main cost of both algorithms besides the
factorizations. Comparing two columns of $n_{solv}$'s in Table
\ref{tab:ex2-comp}, we see that ZoloEig is about 2 to 3 times cheaper than
FEAST in terms of the number of applying the direct solver, $n_{solv}$. More
importantly, when the problem size is large, the factorization time of the
direct solver is dominanting the runtime. In this regime, ZoloEig might be
also more efficient than FEAST since it requires a smaller number of
factorizations. ZoloEig requires only $r=3$ factorizations in Table
\ref{tab:ex2-comp}, while FEAST requires $3$ to $7$ factorizations and the
number of factorizations slightly increases as the problem size grows.

In the case of parallel computing, spectrum slicing algorithms including
both FEAST and ZoloEig could be highly scalable. For example, eigenpairs in
different spectrum ranges can be estimated independently; multishift linear
systems can be solved independently; and each equation solver can be applied
in parallel. If there was unlimited computer resource, then the advantage of
ZoloEig over FEAST in terms of a smaller number of factorization might be
less significant, but still meaningful because the number of iterations
$n_{iter}n_{gmres}$ in ZoloEig (considering both the GMRES iterations and
subspace iterations) is smaller than the number of subspace iterations
$n_{iter}$ in FEAST, and these iteration numbers cannot be reduced by
parallel computing. Therefore, if unlimited computer resource was used, the
total parallel runtime will be dominated by the iteration time in both
ZoloEig and FEAST, and hence ZoloEig could be still faster than FEAST.  Note
that in practice the computer resource might be limited. In such a case, it
is of interest to design faster parallel algorithms with a fixed number of
processes. Given a fixed number of processes, ZoloEig has less number of
matrix factorization and hence can assign more processes to each matrix
factorization and each application of the factorization. Therefore, the
runtime of parallel matrix factorization and iterative part in ZoloEig would
be shorter than that of FEAST. The parallel version of ZoloEig is under
development and it is worth to expore this benifit for large-scale
eigenvalue problems.

\subsection{Florida Sparse Matrix Collection}

In the third example, the proposed algorithm is applied to general sparse
Hermitian matrices from the Florida sparse matrix collection. In order to
show the broad applicability of the algorithm, all Hermitian matrices with
size between 200 and $6,000$ in the collection are tested. The full list of
these matrices can be found in the test file ``test\_eigs\_Florida.m'' in
the MATLAB toolbox. For each of these matrices, we randomly choose an
interval $(a,b)$ containing 96 eigenvalues.

In these examples, we compare the performance of the ZoloEig algorithm with
the FEAST algorithm based on the contour integral method with trapezoidal
rule. The subspace refinement is turned off again, aiming at testing the
approximation accuracies of the Zolotarev's rational function and the 
discretized contour integral. The order $r$ in the Zolotarev's rational
function is $4$ and the contour integral method has $16$ poles. Hence, both
the ZoloEig and FEAST algorithms use the same order of rational functions in
the approximation.

\begin{figure}[htb]
    \centering
    \begin{subfigure}[t]{0.48\textwidth}
        \centering
        \includegraphics[width=\textwidth]{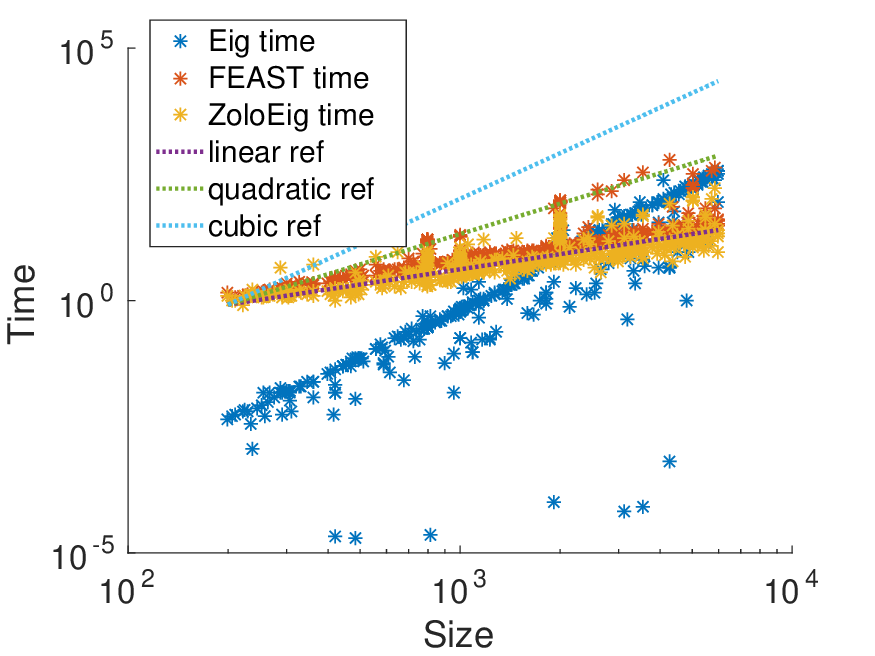}
        \caption{}
        \label{fig:ex2-time}
    \end{subfigure}
    \begin{subfigure}[t]{0.48\textwidth}
        \centering
        \includegraphics[width=\textwidth]{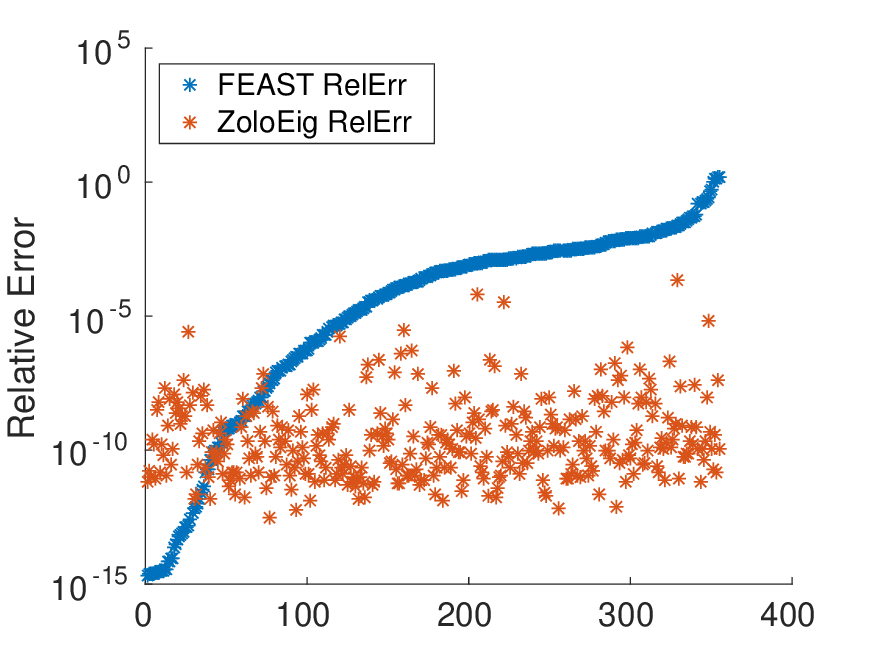}
        \caption{}
        \label{fig:ex2-err}
    \end{subfigure}
    \caption{Running time and relative error for matrices in Florida matrix
    collection solved via ZoloEig and FEAST. The relative error here is the
    relative error of eigenvalues defined in \eqref{eq:relerr}.}
    \label{fig:ex2}
\end{figure} 

Figure~\ref{fig:ex2} visualizes the results of both the ZoloEig and the
FEAST algorithms.  Figure~\ref{fig:ex2-time} includes the total running time
of the MATLAB default dense eigensolver Eig, FEAST and ZoloEig. The running
time of Eig aligns with the cubic scaling reference, whereas the running
times of both FEAST and ZoloEig align with the linear scaling reference. As
explained in previous examples, for a problem of small size, the iterative
part in both FEAST and ZoloEig dominates the running time. The outliers of
each line in these figures are caused by different sparsity densities and
patterns of sparse matrices. According to Figure~\ref{fig:ex2-time}, the
running time of FEAST is constantly larger than ZoloEig. In
Figure~\ref{fig:ex2-err}, the relative error of FEAST is larger than ZoloEig
for most matrices.  Based on the right part of Figure~\ref{fig:ex2-err},
FEAST fails for some sparse matrices, where the relative error is close to
1. Meanwhile, the relative error of ZoloEig is smaller than 1e-4 in all
cases and the overall accuracy is about 1e-10. This observation supports
that the composition of Zolotarev's rational functions is a better way to
approximate rectangular functions.

\section{Conclusion}\label{sec:discussion}

This paper proposed an efficient method for computing selected eigenpairs
of a sparse Hermitian definite matrix pencil $(A,B)$ in the generalized
eigenvalue problem.  First, based on the best rational function
approximations of signum functions by Zolotarev, the best high-order
rational filter in a form of function compositions is proposed. Second,
taking advantage of the shift-invariant property of Krylov subspaces
in iterative methods and the matrix sparsity in sparse direct solvers,
a hybrid fast algorithm is proposed to apply the best rational filter
in the form of function compositions. Assuming that the sparse Hermitian
matrices $A$ and $B$ are of size $N\times N$ and contains $O(N)$ nonzero
entries, the computational cost for computing $O(1)$ eigenpairs is
$O(F_N)$, where $F_N$ is the operation count for solving the shifted
linear system $(A-\sigma B)x=b$ using sparse direct solvers.

Comparing to the state-of-the-art algorithm FEAST, the proposed ZoloEig has 
a better performance in our test examples for sequential computation. The 
numerical results in the sequential computation also implies that ZoloEig might 
also have good performance in parallel computation, which will be left as future 
work.

It is worth pointing out that the proposed rational filter can also
be applied efficiently if an efficient dense direct solver or an
effective iterative solver for solving the multi-shift linear systems
in \eqref{eqn:mts} is available. The proposed rational function
approximation can also be applied as a preconditioner for indefinite
sparse linear system solvers \cite{Xi2017} and the orbital minimization
method in electronic structure calculation \cite{LuYang2017}. These
will be left as future works.

{\bf Acknowledgments.} Y. Li was supported in part by National Science
Foundation under awards DMS-1454939 and OAC-1450280, and also AMS-Simons
Travel Grant. H. Yang was partially supported by the US
National Science Foundation under award DMS-1945029. The authors would like to thank Fabiano
Corsetti for seting up Silicon Bulk examples.

\appendix{Proofs of the properties in Section \ref{sub:matvec}}
\textbf{Proof of Proposition \ref{pro:p1}:}

\begin{proof}
First we prove that we have the following partial fraction representation
\begin{equation}\label{eqn:pb}
    M x \prod_{j = 1}^{r-1} \frac{ x^2 + c_{2j} }{ x^2 + c_{2j-1} } =
    M x \left( 1 + \sum_{j = 1}^{r-1} \frac{b_j}{x^2 + c_{2j-1}} \right),
\end{equation}
where
\begin{equation}
    b_j = (c_{2j} - c_{2j-1}) \prod_{ k = 1, k \neq j}^{r-1} \frac{c_{2k}
    - c_{2j-1}}{c_{2k-1} - c_{2j-1}}
\end{equation}
for $j = 1, \dots, r-1$. Since any rational function has a partial
fraction form and the coefficients $\{c_{2j-1}\}$ are distinct,
Equation~\eqref{eqn:pb} holds. One can verify \eqref{eqn:pb} by
multiplying $x^2 + c_{2j-1}$ to both sides and set $x = \imath
\sqrt{c_{2j-1}}$.

By \eqref{eqn:pb}, we have
\begin{equation}
    Z_{2r}(x; \ell) = M x \frac{ \prod_{j = 1}^{r-1}(x^2 + c_{2j})
    }{ \prod_{j = 1}^{r}(x^2 + c_{2j-1})} = M x \left( 1 + \sum_{j =
    1}^{r-1} \frac{b_j}{x^2 + c_{2j-1}} \right) \frac{1}{x^2 + c_{2r-1}}.
\end{equation}
Hence, simple partial fraction representations of
\begin{equation}
    \frac{b_j}{(x^2 + c_{2j-1})(x^2 + c_{2r-1})} = \frac{b_j}{c_{2r-1} -
    c_{2j-1}} \left( \frac{1}{x^2 + c_{2j-1}} - \frac{1}{x^2 + c_{2r-1}}
    \right)
\end{equation}
for $j = 1, \dots, r-1$ complete the proof of the proposition.
\end{proof}

\textbf{Proof of Proposition \ref{pro:p2}:}

\begin{proof}
    We further decompose \eqref{eqn:pa} as complex rational functions,
    \begin{equation}
	Z_{2r}(x; \ell) = M \sum_{j = 1}^r \frac{a_j}{2} \left(
	\frac{1}{x+\imath \sqrt{c_{2j-1}}} + \frac{1}{x-\imath
    \sqrt{c_{2j-1}}} \right).  \label{eq:zolocomplexpa}
    \end{equation}
    Substitute the M\"obius transformation into \eqref{eq:zolocomplexpa},
    \begin{equation}
        \begin{split}
		Z_{2r}( T(x); \ell) = & M \sum_{j = 1}^r \frac{a_j}{2}
		\left( \frac{x - \beta}{ \gamma (x - \alpha) + \imath
		\sqrt{c_{2j-1}} (x - \beta) } + \frac{x - \beta}{\gamma
		(x-\alpha)-\imath \sqrt{c_{2j-1}}(x - \beta)} \right) \\
	    = & M \sum_{j = 1}^r
	    \frac{a_j}{2} \left( \frac{ \frac{x - \beta}{\gamma +
	    \imath \sqrt{c_{2j-1}} } }{ x - \frac{\gamma \alpha + \imath
	    \sqrt{c_{2j-1}} \beta}{\gamma + \imath \sqrt{c_{2j-1}} } }
	    + \frac{ \frac{x - \beta}{\gamma - \imath \sqrt{c_{2j-1}}
	    } }{ x - \frac{\gamma \alpha - \imath \sqrt{c_{2j-1}}\beta
	    }{\gamma - \imath \sqrt{c_{2j-1}} } } \right).  \\
        \end{split}
	    \label{eq:hZexpan}
    \end{equation}
    We denote
    \begin{equation}
	\sigma_j := \frac{\gamma \alpha + \imath \sqrt{c_{2j-1}} \beta
	}{\gamma + \imath \sqrt{c_{2j-1}} } = \frac{(\gamma^2\alpha +
	c_{2j-1}\beta) + \imath \sqrt{c_{2j-1}}(\beta -
	\alpha)\gamma}{\gamma^2 + c_{2j-1}}.
    \end{equation}
    Readers can verify that
    \begin{equation}
	\bar{\sigma}_j = \frac{\gamma \alpha - \imath \sqrt{c_{2j-1}}\beta
	}{\gamma - \imath \sqrt{c_{2j-1}} },
    \end{equation}
    where $\bar{\sigma}_j$ is the complex conjugate of $\sigma_j$.
    Equation~\eqref{eq:hZexpan} can be rewritten as,
    \begin{equation}
        Z_{2r}( T(x); \ell) = 
        M \sum_{j = 1}^r \frac{a_j \gamma}{\gamma^2 + c_{2j-1}} + M
        \sum_{j = 1}^r \left( \frac{w_j}{x - \sigma_j} +
        \frac{\bar{w}_j}{x - \bar{\sigma}_j}\right),
    \end{equation}
    where 
    \begin{equation}
	w_j = \frac{a_j (\sigma_j - \beta)}{2(\gamma + \imath
	\sqrt{c_{2j-1}})}.
    \end{equation}

\end{proof}

\bibliographystyle{abbrv}
\bibliography{ref}

\end{document}